\documentclass[12pt]{iopart}      %% LaTeX2e document.

\expandafter\let\csname equation*\endcsname\relax
\expandafter\let\csname endequation*\endcsname\relax
\usepackage{amsmath,setspace,graphicx,dsfont,cite,enumerate,color,bm,float,epstopdf,verbatim,subfiles}    
\usepackage{caption}
\usepackage{amssymb, amsthm}
\usepackage[outercaption]{sidecap}
\usepackage{tabularx}
\usepackage[table]{xcolor}
\usepackage{pgffor,lipsum}
\allowdisplaybreaks

\theoremstyle{plain}
\newtheorem{theorem}{Theorem}
\newtheorem*{theorem*}{Theorem}
\newtheorem{lemma}[theorem]{Lemma}
\newtheorem{proposition}[theorem]{Proposition}
\newtheorem{corollary}[theorem]{Corollary}
\newtheorem{definition}[theorem]{Definition}
\newtheorem{remark}[theorem]{Remark}
\newtheorem{remarks}[theorem]{Remarks}
\newtheorem{example}[theorem]{Example}
\newtheorem{assumptions}[theorem]{Assumptions}
\newtheorem{algorithm}[theorem]{Algorithm}

\numberwithin{theorem}{section}
\numberwithin{equation}{section}

\newcommand{\eps}{\varepsilon} 
\newcommand{\dee}{\mathrm{d}}

\newcommand{\intr}{\mathrm{int}}

\graphicspath{{./images/}}

\begin{document}

\makeatother

\newcommand{\op}{\mathrm{op}}
\newcommand{\post}{\mathrm{post}}

\newcommand{\mmd}{}
\newcommand{\new}{}

\title[MAP Estimators for Piecewise Continuous Inversion]{MAP Estimators for Piecewise Continuous Inversion}
\author{M M Dunlop and A M Stuart}
\address{Computing \& Mathematical Sciences, California Institute of Technology, Pasadena, CA 91125, USA}

\eads{\mailto{mdunlop@caltech.edu}, \mailto{astuart@caltech.edu}}

\begin{abstract}
We study the inverse problem of estimating a field $u^a$ from data comprising a finite set of nonlinear functionals of $u^a$, subject to additive noise; we denote this observed data by $y$. Our interest is in the reconstruction of piecewise continuous fields $u^a$ in which the discontinuity set is described by a finite number of geometric parameters $a$. Natural applications include groundwater flow and electrical impedance tomography.  We take a Bayesian approach, placing a prior distribution on $u^a$ and determining the conditional distribution on $u^a$
given the data $y$. It is then natural to study maximum a posterior (MAP)
estimators. Recently (Dashti \emph{et al} 2013 \emph{Inverse Problems} \textbf{29} 095017) it has been shown that MAP estimators can be
characterised as minimisers of a generalised Onsager-Machlup functional, in the case where the prior measure is a Gaussian random field. We extend this theory to a more general class of prior distributions which allows for piecewise continuous fields. Specifically, the prior field is assumed to be piecewise Gaussian with random interfaces between the different Gaussians defined by a finite number of parameters. We also make connections with recent work on MAP estimators for linear problems and possibly non-Gaussian priors (Helin, Burger 2015 \emph{Inverse Problems} \textbf{31} 085009) which employs the notion of Fomin derivative.

In showing applicability of our theory we focus on the groundwater flow and EIT models, though the theory holds more generally. Numerical experiments are implemented for the groundwater flow model, demonstrating the feasibility of determining MAP estimators for these piecewise continuous models, but also that the geometric formulation can lead to multiple nearby (local) MAP estimators. We relate these MAP estimators to the behaviour of output from MCMC samples of the posterior, obtained using a state-of-the-art function space Metropolis-Hastings method.
\end{abstract}

%\submitto{\IP}

\section{Introduction}
\subsection{Context and Literature Review}

A common inverse problem is that of estimating an unknown function from noisy measurements of a (possibly nonlinear) map applied to the function. Statistical and deterministic approaches to this problem have been considered extensively. In this paper we focus on the the study of MAP estimators within the Bayesian approach; these estimators provide a natural link between deterministic and statistical methods. In the Bayesian formulation, we describe the solution probabilistically and the distribution of the unknown, given the measurements and a prior model, is termed the posterior distribution. MAP estimators attempt to work with a notion of solutions of maximal probability under this posterior 
distribution and are typically characterised variationally, linking to deterministic methods.  

There are two main approaches taken to the study of the posterior. The first is to discretise the space, and then apply finite dimensional Bayesian methodology \cite{kaipio}. An advantage to this approach is the availability of a Lebesgue density and a large amount of previous work which can then be built upon; but issues may arise (for example computationally) when the dimension of the discretisation space is increased. An alternative approach is to apply infinite dimensional methodology directly on the original space, to derive algorithms, and then discretise to implement. This approach has been studied for linear problems in \cite{linear1,linear2,linear3}, and more recently for nonlinear problems \cite{lecturenotes,lasanen1,lasanen2,inverse}. It is the latter approach that we focus on in this paper.

In some situations it may be that point estimates are more desirable, or more computationally feasible, than the entire posterior distribution. A detailed study of point estimates can be found in for example\cite{point_est}. Three different estimates are commonly considered: the posterior mean which minimises $L^2$ loss, the posterior median which minimises $L^1$ loss, and posterior modes which minimise zero-one loss. The former two estimates are unique \cite{uniqueness_median}, but a distribution may possess more than one mode. A consequence of this is that the posterior mean and median may be misleading in the case of a multi-modal posterior. Posterior modes are often termed maximum a posteriori (MAP) estimators in the literature. 

In this paper we focus on MAP estimation. If the posterior has Lebesgue density $\rho$, MAP estimators are given by the global maxima of $\rho$. The problem of MAP estimation in this case is hence a deterministic variational problem, and has been well-studied \cite{kaipio}.  In the infinite-dimensional setting there is no Lebesgue density, but there has been recent research aimed at characterising the mode variationally and linking to the classical regularisation techniques described in, for example, \cite{map} in the case when Gaussian priors are adopted. Non-Gaussian priors have also been considered in the infinite dimensional setting -- in \cite{wmap} weak MAP (wMAP) estimators are defined as generalisations of MAP estimators, and a variational characterisation of them is provided in the case that the forward map is linear, using the notion of Fomin derivative.

In this paper we make a significant extension of the work in \cite{map} to include priors which are defined by a combination of Gaussian random fields and a finite number of geometric parameters which define the different domains in which the different random fields apply. We thereby study the reconstruction of piecewise continuous fields with interfaces defined by a finite number of parameters. \mmd{Our motivation for doing so comes from the work in \cite{historymatching}, and its predecessors. In that paper a Bayesian inverse problem for piecewise constant fields, modelling the permeability appearing in a two-phase subsurface flow model, was studied. \new{Such piecewise continuous fields were also previously studied in a groundwater flow context in \cite{kui}, where existence and well-posedness of the posterior distribution were shown.} The idea of single point estimates being misleading is discussed and the existence of multiple local MAP estimators is shown. We also link our work to that in \cite{wmap}, by characterizing the MAP estimator via the Fomin derivative.

Throughout this paper we focus on two model problems: groundwater flow and electrical impedance tomography (EIT). Both of these problems are important examples of large scale inverse problems, \new{with}
applications of great economic and societal value. MAP estimation in such problems has been studied previously \cite{largescale,besovmap,reservoir,frechet}. 
\mmd{However our formulation is quite general; for brevity we simply illustrate
the theory for groundwater flow and EIT, and the numerics only in the
case of groundwater flow.}

\subsection{Mathematical Setting}
\label{ssec:bg}
Let $X$ be a separable Banach space and let $\Lambda \subseteq \mathbb{R}^k$. $X$ should be thought of as a function space and $\Lambda$ a space of geometric parameters. Given $(u,a) \in X\times\Lambda$, we construct another function $u^a \in Z$, say. Considering the ingredients $u$ and $a$ in the construction of this function $u^a$ separately will be useful in what follows. 
\mmd{An example of such a construction is shown in Figure \ref{fig:construction}.}

\begin{figure}
\begin{center}
\includegraphics[width=\textwidth]{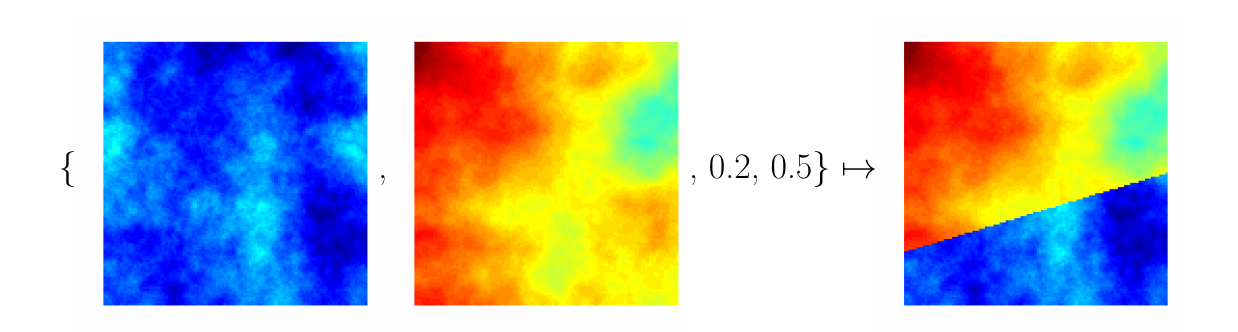}
\end{center}
\caption{\mmd{An example of construction of a piecewise continuous field, using two continuous fields and two scalar parameters. Here the scalar parameters determine the points where the interface meets each side of the domain. We work on the space of continuous fields and parameters, but it is pushforward of these by the construction map that represents the piecewise continuous field we aim to recover.}}
\label{fig:construction}
\end{figure}

Suppose we have a (typically nonlinear) forward operator $\mathcal{G}:X\times\Lambda\rightarrow Y$, where $Y = \mathbb{R}^J$. If $(u,a)$ denotes the true input to our forward problem, we observe data $y \in Y$ given by
\begin{eqnarray*}
y = \mathcal{G}(u,a) + \eta
\end{eqnarray*}
where $\eta \sim N(0,\Gamma)$, $\Gamma \in \mathbb{R}^{J\times J}$ positive definite, is some centred Gaussian noise on $Y$. Modelling everything probabilistically, we build up the joint distribution of $(u,a,y)$ by specifying a prior distribution $\mu_0\times\nu_0$ on $(u,a)$ and an independent noise model on $\eta$. We are then interested in the posterior $\mu$ on $(u,a)$ given $y$. Denote $|\cdot|$ the Euclidean norm on $\mathbb{R}^J$, and for any positive definite $A \in \mathbb{R}^{J\times J}$ denote $|\cdot|_A := |A^{-1/2}\cdot|$ the weighted norm on $\mathbb{R}^J$. Under certain conditions, using a form of Bayes' theorem, we may write $\mu$ in the form
\begin{eqnarray*}
\mu(\dee u, \dee a) \propto \exp\left(-\frac{1}{2}|\mathcal{G}(u,a)-y|_{\Gamma}^2\right)\mu_0(\dee u)\nu_0(\dee a).
\end{eqnarray*}

\mmd{The modes of the posterior distribution, termed MAP (maximum a posteriori) estimators, can be considered `best guesses' for the state $(u,a)$ given the data $y$. 
%Our main aim is to show that finding MAP estimators of $\mu$ is the same as finding minimisers of a certain functional, called the Onsager-Machlup functional, over a (small) subset $E\times S \subseteq X\times\Lambda$.
We now state rigorously what we mean by a MAP estimator for $\mu$, as in \cite{map}. Given $(u,a) \in X\times\Lambda$, denote by $B^\delta(u,a)$ the ball of radius $\delta$ centred at $(u,a)$.
\begin{definition}[MAP estimator]
For each $\delta > 0$, define
\begin{eqnarray*}
(u^\delta,a^\delta) = \underset{(u,a)\in X\times\Lambda}{\mathrm{argmax}} \mu(B^\delta(u,a)).
\end{eqnarray*}
Any point $(\bar{u},\bar{a}) \in X\times\Lambda$ satisfying
\begin{eqnarray*}
\lim_{\delta\downarrow 0}\frac{\mu(B^\delta(\bar{u},\bar{a}))}{\mu(B^\delta(u^\delta,a^\delta))} = 1
\end{eqnarray*}
is called a MAP estimator for the measure $\mu$.
\end{definition}
If this definition is applied to probability measures defined via a Lebesgue density, MAP estimators coincide with maxima of this density. Here we extend
the notion to the study of piecewise continuous fields.}

\subsection{Our Contribution}

\mmd{The primary contributions of the paper are fourfold:}

\begin{enumerate}[(i)]

\item \mmd{We develop the MAP estimator theory for infinite dimensional
geometric inverse problems involving discontinuous fields, building
on theory in both of the  recent papers \cite{map,wmap}, and opening up new avenues
for the study of MAP estimators in infinite dimensional inverse problems.
}

\item \mmd{We explicitly link MAP estimation for these geometric
inverse problems to a variational Onsager-Machlup minimization problem.
}

\item We show that the theory applies to groundwater flow model as in \cite{kui} and we show that the theory applies to the EIT problem as in \cite{eit}.

\item We implement numerical experiments for the groundwater flow model and demonstrate the feasibility of computing (local) MAP estimators within the geometric formulation, but also show that they can lead to multiple nearby solutions. We relate these multiple MAP estimators to the behaviour of output from MCMC to probe the posterior.
\end{enumerate}
\subsection{Structure of the Paper}
\begin{itemize}
\item In section \ref{sec:fwd} we describe the forward maps associated with the groundwater flow and EIT problems, and show that they have the appropriate regularity needed in sections \ref{sec:post}--\ref{sec:map}. 
\item In section \ref{sec:prior} we describe the choice of, and assumptions upon, the prior distribution whose samples comprise piecewise Gaussian random fields 
with random interfaces. 
\item In section \ref{sec:post} we show existence and uniqueness of the posterior distribution. 
\item In section \ref{sec:map} we define MAP estimators and prove their equivalence to minimisers of an appropriate Onsager-Machlup functional.
\item In section \ref{sec:num} we present numerics for the groundwater flow problem. We consider three different prior models and investigate maximisers of the posterior distribution.
\item \mmd{In section \ref{sec:concl} we conclude and outline possible future work in the area.}
\end{itemize}
\section{The Forward Problem}
\label{sec:fwd}
We consider two model problems. Our first problem (groundwater flow) is that of determining the piecewise continuous permeability of a medium, given noisy measurements of water pressure (or hydraulic head) within it. The second problem (EIT) is determination of the piecewise continuous conductivity within a body from boundary voltage measurements.

In what follows, the finite dimensional space $\Lambda$ will be a space of geometric parameters defining the interfaces between different media, and $X$ will be a product of function spaces defining the values of the permeabilities/conductivities between the interfaces.

We begin in subsection \ref{ssec:interface} by defining the construction map $(u,a)\mapsto u^a$ for the piecewise continuous fields. In subsections \ref{ssec:gwf} and \ref{ssec:eit} we describe the models for groundwater flow and EIT respectively, and prove regularity properties of the resulting forward maps; these properties are required for our subsequent theory.

\subsection{Defining the Interfaces}
\label{ssec:interface}
Let $D \subseteq \mathbb{R}^d$ be the domain of interest and let $\Lambda \subseteq \mathbb{R}^k$ be the space of geometric parameters. Take a collection of set-valued maps $A_i:\Lambda\rightarrow\mathcal{B}(D)$, $i=1,\ldots,N$ such that for each $a \in \Lambda$ we have
\begin{eqnarray*}
\bigcup_{i=1}^N{A_i(a)} = D,\;\;\;A_i(a)\cap A_j(a) = \varnothing\text{ if }i \neq j.
\end{eqnarray*}
We assume that each map $A_i$ is continuous in the sense that
\begin{eqnarray*}
|a - b|\rightarrow 0 \Rightarrow |A_i(a)\Delta A_i(b)|\rightarrow 0
\end{eqnarray*}
where $\Delta$ denotes the symmetric difference:
\begin{eqnarray*}
A\Delta B := (A\setminus B)\cup(B\setminus A).
\end{eqnarray*}

Let $X = C^0(D;\mathbb{R}^N)$. Given $u = (u_1,\ldots,u_N) \in X$ and $a \in \Lambda$ we define the function $u^a \in L^\infty(D)$ by
\begin{eqnarray}
\label{eq:construction}
u^a = F(u,a) := \sum_{i=1}^N u_i\mathds{1}_{A_i(a)}.
\end{eqnarray}
where $F:X\times\Lambda\rightarrow L^\infty(D)$ is the construction map.

We give four examples of the functions $A_i$ and the sets/interfaces they define.

\begin{example}
\label{ex:straight}
Let $D = [0,1]^2$, $\Lambda = [0,1]^2$ and $N = 2$. We specify points $a$ and $b$ on either side of the square $D$ and join them with a straight line. We then let $A_1(a,b)$ be the region of $D$ below this line and $A_2(a,b) = D\setminus A_1(a,b)$.
\end{example}

\begin{figure}[h!]
\begin{center}
\includegraphics[width=\textwidth]{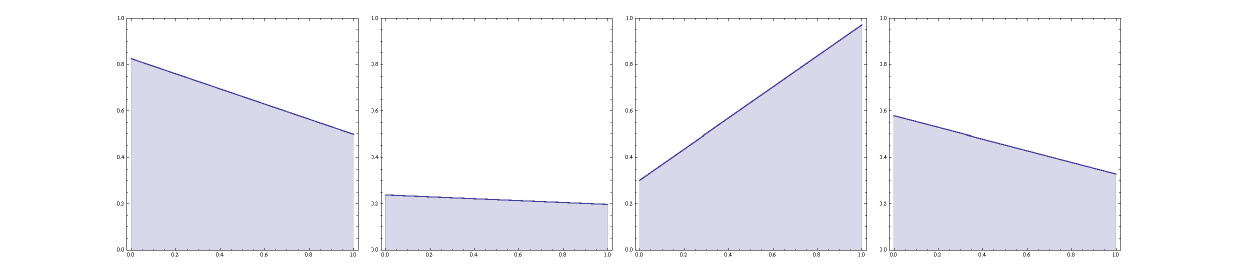}
\end{center}
\caption{Possible sets $A_i$ corresponding to Example \ref{ex:straight}}
\end{figure}

\begin{example}
\label{ex:curve}
Let $D = [0,1]^2$, $\Lambda = [0,1]^2$ and $N = 2$. Choose a continuous map $H:\Lambda\rightarrow L^{\infty}([0,1])$ such that $H(a,b)(0) = a$ and $H(a,b)(1) = b$ for all $(a,b) \in \Lambda$. Let $A_1(a,b)$ be the region of $D$ beneath the graph of the curve $H(a,b)$ and let $A_2(a,b) = D\setminus A_1(a,b)$. This setup includes the previous example: $H(a,b)(x) = a + (b-a)x$ defines the appropriate straight lines.

The continuity of $A_1$ and $A_2$ can be seen by noting that
\begin{eqnarray*}
|A_1(a_1,b_1)\Delta A_1(a_2,b_2)| &= |A_2(a_1,b_1)\Delta A_2(a_2,b_2)|\\
&\leq\int_0^1 |H(a_1,b_1)(x) - H(a_2,b_2)(x)|\,\dee x\\
&\leq \|H(a_1,b_1) - H(a_2,b_2)\|_{\infty}
\end{eqnarray*}
and using the continuity of $H$ into $L^\infty([0,1])$.

For example, one may take $H$ to be given by
\begin{eqnarray*}
H(a,b)(x) = a + (b-a)x + x\sin(6\pi x)/10
\end{eqnarray*}
which can be seen to be continuous into $L^\infty([0,1])$.
\end{example}

\begin{figure}[h!]
\begin{center}
\includegraphics[width=\textwidth]{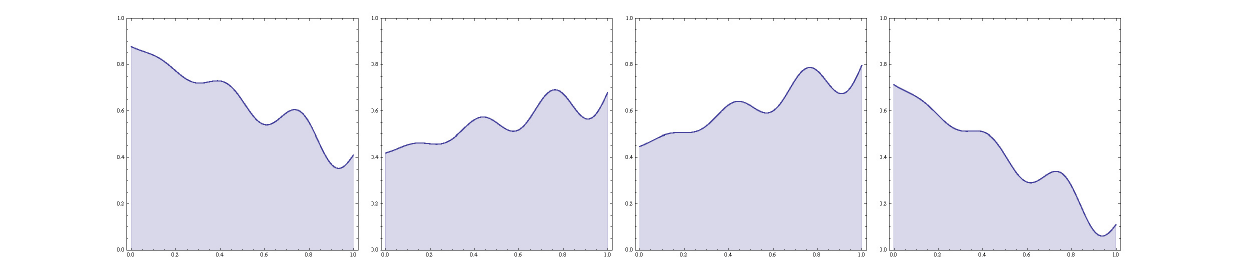}
\end{center}
\caption{Possible sets $A_i$, corresponding to Example \ref{ex:curve}}
\end{figure}

\mmd{\begin{example}
\label{ex:fault}
We can generalise the previous example to allow the inclusion of a fault. Let $D = [0,1]^2$, $\Lambda = [0,1]^2\times[-1,1]$ and $N=2$. Let $p \in (0,1)$ denote the horizontal location of the fault. Given $H:[0,1]^2\rightarrow L^\infty([0,1])$ as in the previous example, define $\tilde{H}:\Lambda\rightarrow L^\infty([0,1])$ by
\[
\tilde{H}(a,b,c)(x) =
\begin{cases}
H(a,b)(x) & x \in [0,p]\\
c + H(a,b)(x) & x \in (p,1]
\end{cases}
\]
so that the parameter $c$ determines the (signed) magnitude of the fault. Defining the sets $A_1(a,b,c)$ and $A_2(a,b,c)$ as the regions of $D$ beneath and above the curve $\tilde{H}(a,b,c)$ respectively, the continuity can be seen in a similar manner to the previous example.
\end{example}}

\begin{figure}[h!]
\begin{center}
\includegraphics[width=\textwidth]{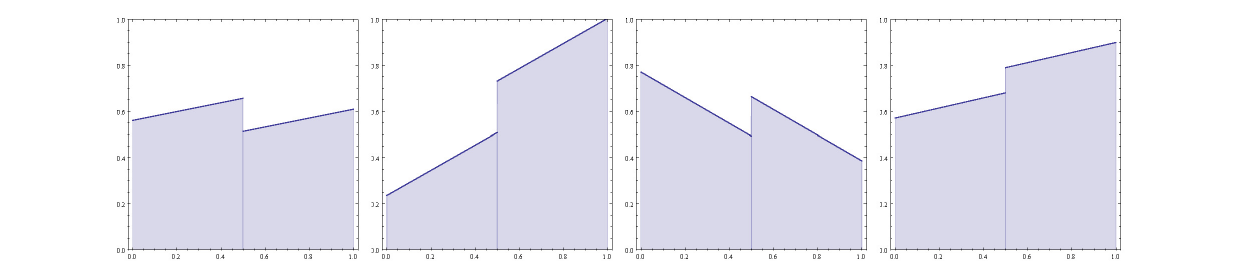}
\end{center}
\caption{\mmd{Possible sets $A_i$, corresponding to Example \ref{ex:fault} in the case $p=1/2$.}}
\end{figure}

\begin{example}
\label{ex:many}
Again working with $D = [0,1]^2$, but with a much larger parameter space, one could also select points at specific $x$-coordinates and linearly interpolate between them. Fix $K, N \in \mathbb{N}$ and set $\Lambda = \Xi_{N-1}^K \subseteq [0,1]^{(N-1)\times K}$, where $\Xi_{N-1}$ is the simplex
\begin{eqnarray*}
\Xi_{N-1} = \{(y_1,\ldots,y_{N-1})\in [0,1]^{N-1}\;|\;0\leq y_1\leq\ldots\leq y_{N-1}\leq 1\}.
\end{eqnarray*} 
Then given $a \in \Lambda$, define the functions $f_i(a)$, $i=1,\ldots,N-1$, to be the linear interpolation of the points $\big(\frac{j-1}{K-1},a_{ij}\big)_{j=1}^{K}$. $A_i(a)$, $i=1,\ldots,N-1$, is then defined to be the region between the graphs of the functions $f_i(a)$ and $f_{i-1}(a)$, and $A_N(a) = D\setminus \cup_{i=1}^{N-1} A_{i}(a)$.

In order to see the continuity of these maps, we first partition the domain into strips $D_j$,
\begin{eqnarray*}
D_j = \left\{(x,y) \in D\;\bigg|\; \frac{j-1}{K-1} \leq x \leq \frac{j}{K-1}\right\},\;\;\;j=1,\ldots,K-1
\end{eqnarray*}
so that we have
\begin{eqnarray*}
A_i(a) = \bigcup_{j=1}^{K-1}A_i(a)\cap D_j.
\end{eqnarray*}
It follows from properties of the symmetric difference that
\begin{eqnarray*}
|A_i(a)\Delta A_i(b)| \leq \sum_{j=1}^{K-1}|(A_i(a)\cap D_j)\Delta (A_i(b)\cap D_j)|.
\end{eqnarray*}
It hence suffices to show that the maps $A_i(\cdot)\cap D_j$ are continuous for all $i,j$. This follows from the same argument as in Example \ref{ex:curve}, for sufficiently small $|a-b|$.
\end{example}

\begin{figure}[h!]
\begin{center}
\includegraphics[width=\textwidth]{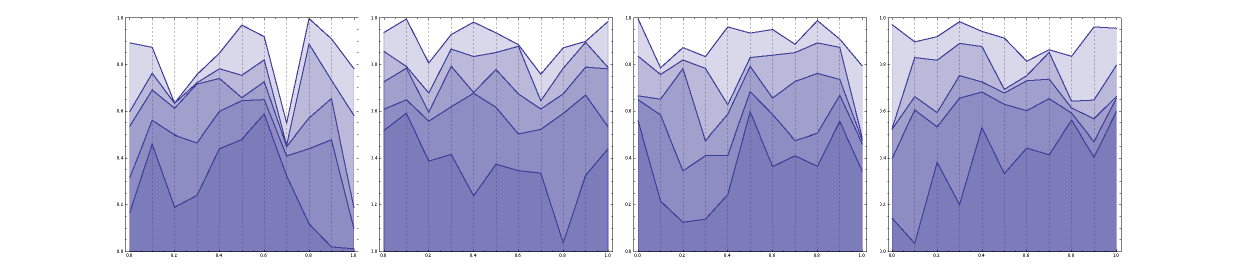}
\end{center}
\caption{Possible sets $A_i$, corresponding to Example \ref{ex:many} in the case $K=11$, $N=6$}
\end{figure}

\subsection{The Darcy Model for Groundwater Flow}
\label{ssec:gwf}
We consider the Darcy model for groundwater flow on a domain $D\subseteq \mathbb{R}^d$, $d=1,2,3$. Let $\kappa = (\kappa_{ij})$ denote the permeability tensor of the medium, $p$ the pressure of the water, and assume the viscosity of the water is constant. Darcy's law \cite{darcy} tells us that the velocity is proportional to the gradient of the pressure:
\begin{eqnarray*}
v = -\kappa\nabla p.
\end{eqnarray*}
Additionally, a local form of mass conservation tells us that
\begin{eqnarray*}
\nabla \cdot v = f.
\end{eqnarray*}
Combining these two equations, and imposing Dirichlet boundary conditions for simplicity, results in the PDE
\begin{eqnarray*}
\begin{cases}
-\nabla\cdot(\kappa\nabla p) = f & \text{ in } D\\
\hspace{1.89cm}p = g &\text{ on } \partial D.
\end{cases}
\end{eqnarray*}
This is the PDE we will consider in the forward model, and it gives rise to a solution map $\kappa\mapsto p$.

For simplicity we will work in the case where $\kappa$ is an isotropic (scalar) permeability, bounded above and below by positive constants, and so it can be represented as the image of some bounded function under a positive continuously differentiable map $\sigma:\mathbb{R}\rightarrow\mathbb{R}^+$. 

Let $V = H^1(D)$, the Sobolev space of once weakly differentiable functions on $D$ \cite{gilbarg}. Then given $f \in H^{-1}(D)$, $g \in H^{1/2}(\partial D)$, $u \in X$ and $a \in \Lambda$, define $p_{u,a} \in V$ to be the solution of the weak form of the PDE
\begin{eqnarray}
\label{pde}
\begin{cases}
-\nabla\cdot(\sigma(u^a)\nabla p_{u,a}) &= f\;\;\text{ in }D\\
\hspace{2.6cm}p_{u,a} &= g\;\,\,\text{ on }\partial D.
\end{cases}
\end{eqnarray}
We are first interested in the regularity of the map $\mathcal{R}:X\times\Lambda\rightarrow V$ given by $\mathcal{R}(u,a) = p_{u,a}$. We first recall what it means for $p_{u,a}$ to be a solution of (\ref{pde}). Since $g \in H^{1/2}(\partial D)$, by the trace theorem \cite{gilbarg} there exists $G \in V$ such that $\tr(G) = g$. The solution $p_{u,a}$ of (\ref{pde}) is then given by $p_{u,a} = q_{u,a} + G$, where $q_{u,a} \in H^1_0(D)$ solves the PDE
\begin{eqnarray}
\label{pde2}
\begin{cases}
-\nabla\cdot\big(\sigma(u^a)\nabla q_{u,a}\big) &= f + \nabla\cdot\big(\sigma(u^a)\nabla G\big)\;\;\text{ in }D\\
\hspace{2.6cm}q_{u,a} &= 0\hspace{3.5cm}\text{ on }\partial D.
\end{cases}
\end{eqnarray}
The following lemma tells us that the map $\mathcal{R}$ is well defined and has certain regularity properties. Its proof is given in the appendix.

\begin{lemma}
\label{lem:cont_darcy}
The map $\mathcal{R}:X\times\Lambda \rightarrow V$ is well-defined and satisfies:
\begin{enumerate}[(i)]
\item for each $(u,a) \in X\times\Lambda$,
\begin{eqnarray*}
\|\mathcal{R}(u,a)\|_V \leq (\|f\|_{V^*} + \|\sigma(u^a)\|_{L^\infty}\|G\|_V)/\kappa_{\min}(u,a) + \|G\|_V
\end{eqnarray*}
where $\kappa_{\min}(u,a)$ is given by
\begin{eqnarray*}
\kappa_{\min}(u,a) = \underset{x\in D}{\mathrm{essinf}}\,\sigma(u^a(x)) > 0;
\end{eqnarray*}
\item for each $a \in \Lambda$, $\mathcal{R}(\cdot,a):X\rightarrow V$ is locally Lipschitz continuous, i.e. for every $r > 0$ there exists $L(r) > 0$ such that, for all $u,v \in X$ with $\|u\|_{X}, \|v\|_{X} < r$ and all $a \in \Lambda$, we have
\begin{eqnarray*}
\|\mathcal{R}(u,a) - \mathcal{R}(v,a)\|_V \leq L(r)\|u-v\|_{X};
\end{eqnarray*}
\item for each $u \in X$, $\mathcal{R}(u,\cdot):\Lambda\rightarrow V$ is continuous.
\end{enumerate}

\end{lemma}

We now choose a continuous linear observation operator $\ell:V\rightarrow\mathbb{R}^J$. For example, writing $\ell = (\ell_1,\ldots,\ell_J)$, we could take
\begin{eqnarray}
\label{eq:obs_op}
\ell_i(p) = \int_D \frac{1}{(2\pi\eps)^{d/2}}e^{-|x_i-y|^2/2\eps}p(y)\,\dee x,\;\;\;i=1,\ldots,J
\end{eqnarray}
for some $\eps > 0$, so that $\ell_i$ approximates a point observation at the point $x_i \in D$. Our forward operator $\mathcal{G}:X\times\Lambda\rightarrow\mathbb{R}^J$ is then defined by $\mathcal{G} = \ell \circ \mathcal{R}$, so that it can be written as the composition
\begin{eqnarray*}
(u,a)\mapsto u^a \mapsto \kappa = \sigma(u^a) \mapsto p \mapsto \ell(p)
\end{eqnarray*}

From the above regularity of $\mathcal{R}$ we can deduce the following regularity properties of our forward operator $\mathcal{G}$:
\begin{proposition}
\label{cor:cont_gwf}
Define the map $\mathcal{G}:X\times\Lambda\rightarrow\mathbb{R}^J$ as above. Then $\mathcal{G}$ satisfies
\begin{enumerate}
\item For each $r > 0$ and $u,v \in X$ with $\|u\|_X,\|v\|_X < r$, there exists $C(r) > 0$ such that for all $a \in \Lambda$,
\begin{eqnarray*}
|\mathcal{G}(u,a) - \mathcal{G}(v,a)| \leq C(r)\|u-v\|_X.
\end{eqnarray*}
\item For each $u \in X$, the map $\mathcal{G}(u,\cdot):\Lambda\rightarrow\mathbb{R}^J$ is continuous.
\end{enumerate}
\end{proposition}

\begin{proof}
\begin{enumerate}
\item The map $\ell$ is defined to be a continuous linear functional, and so in particular is Lipschitz. Since we have $\mathcal{G} = \ell\circ\mathcal{R}$ the result follows from Lemma \ref{lem:cont_darcy}(ii).
\item This follows from the continuity of $\ell$ and Lemma \ref{lem:cont_darcy}(iii).
\end{enumerate}
\end{proof}

\subsection{The Complete Electrode Model for EIT}
\label{ssec:eit}
Electrical Impedance Tomography (EIT) is an imaging technique that aims to make inference about the internal conductivity of a body from surface voltage measurements. Electrodes are attached to the surface of the body, current is injected, and the resulting voltages on the electrodes are measured. Applications include both medical imaging, where the aim is to non-invasively detect internal abnormalities within a human patient, and subsurface imaging, where material properties of the subsurface are differentiated via their conductivities. Early references include \cite{eit78} in the context of medical imaging and \cite{ert33} in the context of subsurface imaging.  

The complete electrode model (CEM) is proposed for the forward model in \cite{cheney}, and shown to agree with experimental data up to measurement precision. In its strong form, the PDE reads
\begin{eqnarray}
\label{eq:eitpde}
\begin{cases}
\displaystyle-\nabla\cdot(\kappa(x)\nabla v(x)) = 0 & x \in D\\[.8em]
\displaystyle\int_{e_l} \kappa\frac{\partial v}{\partial n}\,\dee S = I_l & l=1,\ldots,L\\[.8em]
\displaystyle\kappa(x)\frac{\partial v}{\partial n}(x) = 0 & x \in \partial D\setminus\bigcup_{l=1}^L e_l\\[.8em]
\displaystyle v(x) + z_l\kappa(x)\frac{\partial v}{\partial n}(x) = V_l & x \in e_l, l=1,\ldots,L. 
\end{cases}
\end{eqnarray}

\begin{figure}
\begin{center}
\includegraphics[width=0.5\textwidth, trim=0cm 2cm 0cm 0cm]{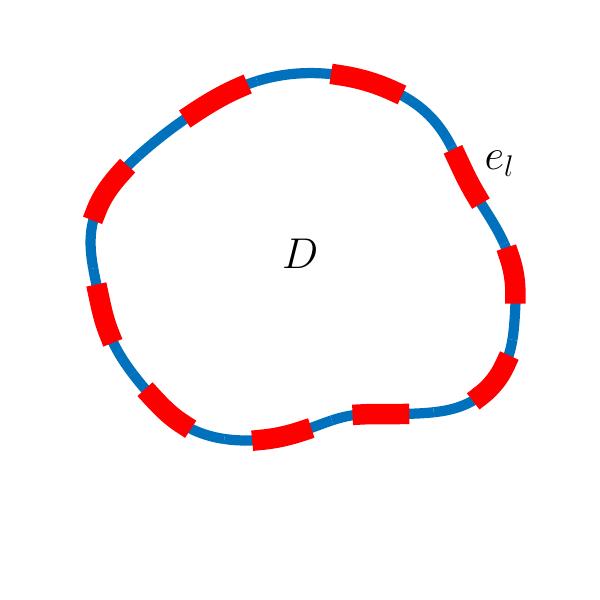}
\caption{An example domain $D$, with attached electrodes $(e_l)_{l=1}^L$, for the EIT problem.}
\label{fig:eit_domain}
\end{center}
\end{figure}

The domain $D$ represents the body, and $(e_l)_{l=1}^L\subseteq \partial D$ the electrodes attached to its surface with corresponding contact impedances $(z_l)_{l=1}^L$. A current $I_l$ is injected into each electrode $e_l$, and a voltage measurement $V_l$ made. Here $\kappa$ represents the conductivity of the body, and $v$ the potential within it. Note that the solution comprises both a function $v \in H^1(D)$ and a vector $(V_l)_{l=1}^L \in \mathbb{R}^L$ of boundary voltage measurements.

A corresponding weak form exists, and is shown to have a unique solution (up to constants) given appropriate conditions on $\kappa$, $(z_l)_{l=1}^L$ and $(I_l)_{l=1}^L$ -- see \cite{cheney} for details. Moreover, under some additional assumptions, the mapping $\kappa\mapsto (V_l)_{l=1}^L$ is known to be Fr\'echet differentiable when we equip the conductivity space with the supremum norm \cite{frechet}.

We can apply different current stimulation patterns to the electrodes to yield additional information. Assume that we have $M$ different (linearly independent) current stimulation patterns $(I^{(m)})_{m=1}^M$. This yields $M$ different mappings $\kappa\mapsto (V_l^{(m)})_{l=1}^L$ each with the regularity above, or equivalently a mapping $\kappa\mapsto V$ where $V \in \mathbb{R}^{J}$ with $J = LM$.

Analogously to the Darcy model case, we will consider isotropic conductivities of the form $\kappa = \sigma(u^a)$, where $\sigma:\mathbb{R}\rightarrow\mathbb{R}^+$ is positive and continuously differentiable. Our forward operator $\mathcal{G}:X\times\Lambda\rightarrow \mathbb{R}^J$, is then given by the composition
\begin{eqnarray*}
\fl(u,a) \mapsto u^a \mapsto \kappa = \sigma(u^a) \mapsto \big((v^{(1)},V^{(1)}),\ldots,(v^{(M)},V^{(M)})\big) \mapsto (V^{(1)},\ldots,V^{(M)}).
\end{eqnarray*}
We show in the appendix that the map defined in this way has the same regularity as the map corresponding to the Darcy model.

\begin{proposition}
\label{prop:cont_eit}
Define the map $\mathcal{G}:X\times\Lambda\rightarrow\mathbb{R}^J$ as above. Then $\mathcal{G}$ satisfies
\begin{enumerate}
\item For each $r > 0$ and $u,v \in X$ with $\|u\|_X,\|v\|_X < r$, there exists $C(r) > 0$ such that for all $a \in \Lambda$,
\begin{eqnarray*}
|\mathcal{G}(u,a) - \mathcal{G}(v,a)| \leq C(r)\|u-v\|_X.
\end{eqnarray*}
\item For each $u \in X$, the map $\mathcal{G}(u,\cdot):\Lambda\rightarrow\mathbb{R}^J$ is continuous.
\end{enumerate}
\end{proposition}

\section{Onsager-Machlup Functionals and Prior Modelling}
\label{sec:prior}

\mmd{In this section we recall the definition of an Onsager-Machlup functional for a measure which is equivalent\footnote{\mmd{Two measures $\nu,\mu$ on a measurable space $(M,\mathcal{M})$ are equivalent if $\nu(A) = 0$ if and only if $\mu(A) = 0$, for $A \in \mathcal{M}$.}} to a Gaussian measure. We then introduce the prior measures that we will consider, first on the function space $X$, then the geometric parameter space $\Lambda$, and finally the product space $X\times\Lambda$. We conclude the section by extending the definition of Onsager-Machlup functional so that it is appropriate for the measures we consider here, supported on fields and geometric parameters which are combined
to make piecewise continuous functions.} 

\subsection{Onsager-Machlup Functionals}
The Onsager-Machlup functional of a measure is the negative logarithm of its Lebesgue density when such a density exists, and otherwise can be thought of analogously. We start by defining it precisely for measures defined via density with respect to a Gaussian, allowing for infinite dimensional spaces on which Lebesgue measure is not defined. Suppose that $\mu$ is a measure equivalent to a Gaussian measure $\mu_0$. Then the Onsager-Machlup functional for $\mu$ is defined as follows.

\begin{definition}[Onsager-Machlup functional I]
Let $\mu$ be a measure on a Banach space $Z$ which is equivalent to $\mu_0$, where $\mu_0$ is a Gaussian measure on $Z$ with Cameron-Martin space $E$. Let $B^\delta(z)$ denote the ball of radius $\delta$ centred at $z \in Z$. A functional $I:Z\rightarrow\overline{\mathbb{R}}$ is called the \emph{Onsager-Machlup functional} for $\mu$ if, for each $x,y \in E$,
\begin{eqnarray*}
\lim_{\delta\downarrow 0}\frac{\mu(B^\delta(x))}{\mu(B^\delta(y))} = \exp\left(I(y)-I(x)\right)
\end{eqnarray*}
and $I(x) = \infty$ for $x\notin E$.
\end{definition}

\begin{remarks}
\begin{enumerate}[(i)]
\item The Onsager-Machlup functional is only defined up to addition of a constant.
\item  If $Z$ is finite dimensional and $\mu$ admits a positive Lebesgue density $\rho$, then $I(x) = -\log\rho(x)$ for all $x \in Z$. In light of the previous remark, this is true even if $\rho$ is not normalised.
\item Let $Z = \mathbb{R}^n$ be finite dimensional, and let $\mu_0 = N(0,\Sigma)$ be a Gaussian measure on $Z$. Let $\Gamma \in \mathbb{R}^{m\times m}$ be a positive-definite matrix, $A \in \mathbb{R}^{m\times n}$ and $y \in \mathbb{R}^m$. Define $\mu$ by 
\begin{eqnarray*}
\frac{\dee \mu}{\dee \mu_0}(x) \propto \exp\left(-\frac{1}{2}|Ax-y|_\Gamma^2\right)
\end{eqnarray*}
so that
\begin{eqnarray*}
\frac{\dee \mu}{\dee x}(x) \propto \exp\left(-\frac{1}{2}|Ax-y|_\Gamma^2 - \frac{1}{2}|x|_{\Sigma}^2\right).
\end{eqnarray*}
Then by the previous remark, the Onsager-Machlup functional for $\mu$ is given by
\begin{eqnarray*}
I(x) = \frac{1}{2}|Ax-y|_\Gamma^2 + \frac{1}{2}|x|_{\Sigma}^2
\end{eqnarray*}
for all $x \in Z$, which is a Tikhonov regularised least squares functional.
\item The preceding example (iii) may be extended to an infinite dimensional setting. Let $Z$ be a separable Banach space, and let $\mu_0 = N(0,\mathcal{C}_0)$ be a Gaussian measure on $Z$ with Cameron-Martin space $(E,\langle\cdot,\cdot\rangle_E,\|\cdot\|_E)$. Let $\Gamma \in \mathbb{R}^{m\times m}$ be a positive-definite matrix, $A:X\rightarrow \mathbb{R}^m$ a bounded linear operator and $y \in \mathbb{R}^m$. Define $\mu$ by
\begin{eqnarray*}
\frac{\dee \mu}{\dee \mu_0}(x) \propto \exp\left(-\frac{1}{2}|Ax-y|_\Gamma^2\right).
\end{eqnarray*}
Then Theorem 3.2 in \cite{map} tells us that the Onsager-Machlup functional for $\mu$ is given by\
\begin{eqnarray*}
I(x) = \frac{1}{2}|Ax - y|_\Gamma^2 + \frac{1}{2}\|x\|_E^2.
\end{eqnarray*}
\item In this paper, the posterior distribution will be a measure on the product space $Z = X\times\Lambda$. The prior distribution will be an independent product of a Gaussian on $X$ and a compactly supported measure on $\Lambda$. Due to the assumption of compact support, the prior will not be equivalent to a Gaussian measure on $Z$ and so the above definition doesn't apply; we provide a suitable extension to the definition in subsection \ref{ssec:prior}.
\end{enumerate}
\end{remarks}

As we are taking a Bayesian approach to the inverse problem, we incorporate our prior beliefs about the permeability/conductivity into the model via probability measures on $X$ and $\Lambda$. We will combine these into a prior measure on the product space $X\times\Lambda$. We equip this space with any (complete) norm $\|(\cdot,\cdot)\|$ such that if $\|(u,a)\|\rightarrow 0$, then $\|u\|_X\rightarrow 0$ and $|a|\rightarrow 0$.

\subsection{Priors for the Fields}
\label{ssec:priorfield}

We wish to put priors on the fields $u_1,\ldots,u_N \in C^0(D)$. We use independent Gaussian measures $u_i \sim \mu_0^i := N(m_i,\mathcal{C}_i)$, where the means $m_i \in C^0(D)$, and each covariance operator $\mathcal{C}_i:C^0(D)\rightarrow C^0(D)$ is trace-class and positive definite. It follows that the vector $(u_1,\ldots,u_N)\sim \mu_0^1\times\ldots\times\mu_0^N =:\mu_0$ is Gaussian on $X$:
\begin{eqnarray*}
\mu_0 = N\left(m,\bigoplus_{i=1}^N \mathcal{C}_i\right)
\end{eqnarray*}
where $m = (m_1,\ldots,m_N) \in X$. If $E_i$ denotes the Cameron-Martin space \cite{lecturenotes} of $\mu_0^i$, then that of $\mu_0$ is given by 
\begin{eqnarray*}
E = \bigoplus_{i=1}^N E_i
\end{eqnarray*}
with inner product given by the sum of those of its component spaces.

The Onsager-Machlup functional of $\mu_0$ is known to be given by
\begin{eqnarray*}
J(u) = 
\begin{cases}
\frac{1}{2}\|u - m\|^2_E & u - m \in E\\
\infty &u - m \notin E.
\end{cases}
\end{eqnarray*}
This can be seen, for example, as a consequence of Proposition 18.3 in \cite{lifshits}.

\begin{remark}
We may assume that the different fields are correlated under the prior, so long as $\mu_0$ remains Gaussian on $X$ -- this does not affect any of the following theory. Allowing correlations between the fields and the geometric parameters under the prior is a more technical issue however, and so we will assume that these are independent. 
\end{remark}

\mmd{\begin{example}
\label{ex:laplacian}
Define the negative Laplacian with Neumann boundary conditions as follows:
\begin{eqnarray*}
A = -\Delta,\;\;\;\mathcal{D}(A) &= \left\{u \in H^2(D)\;\bigg|\;\frac{\dee u}{\dee \nu} = 0\text{ on }\partial D, \int_Du(x)\,\dee x = 0\right\}.
\end{eqnarray*}
Then $A$ is invertible. We can define $\mathcal{C}_i = A^{-\alpha_i}$, where each $\alpha_i > d/2$. Then each $\mathcal{C}_i$ is trace-class and positive definite, and samples from each $\mu_0^i$ will be almost surely continuous and so $\mu_0$ can be considered as a Gaussian measure on $X$. Moreover, regularity of the samples will increase as $\alpha_i$ increases, see \cite{lecturenotes} for details.
\end{example}}

\subsection{Priors for the Geometric Parameters}
\label{ssec:priorgeo}

We also want to put a prior measure on the geometric parameters, i.e. we want to choose a probability measure on $\Lambda$. Since $\Lambda \subseteq \mathbb{R}^k$ the analysis is more straightforward than the infinite dimensional case. Let $\nu$ be a probability measure on $\Lambda$ with compact support $S \subseteq \Lambda$. We assume $\nu$ is absolutely continuous with respect to the Lebesgue measure and that its density $\rho$ is continuous on $S$. Despite being defined on a finite dimensional space, the measure $\nu$ is not necessarily equivalent to the Lebesgue measure on the whole of $\mathbb{R}^k$ and so the previous definition of Onsager-Machlup functional does not apply. We hence must formulate a new definition for this case.

Since $\rho > 0$ on $\intr(S)$, we can use the continuity of $\rho$ to calculate the limits of ratios of small ball probabilities for $\nu$ on $\intr(S)$. Let $a_1,a_2 \in \intr(S)$, then
\begin{eqnarray*}
\lim_{\delta\downarrow 0}\frac{\nu(B^{\delta}(a_1))}{\nu(B^{\delta}(a_2))} &= \lim_{\delta\downarrow 0}\frac{\int_{B^{\delta}(a_1)} \rho(a)\,\dee a}{\int_{B^{\delta}(a_2)} \rho(a)\,\dee a}\\
&= \lim_{\delta\downarrow 0}\frac{\frac{1}{|B^{\delta}(a_1)|}\int_{B^{\delta}(a_1)} \rho(a)\,\dee a}{\frac{1}{|B^{\delta}(a_2)|}\int_{B^{\delta}(a_2)} \rho(a)\,\dee a}\\
&= \frac{\rho(a_1)}{\rho(a_2)}\\
&= \exp\left(\log\rho(a_1) - \log\rho(a_2)\right).
\end{eqnarray*}
If either $a_1$ or $a_2$ lie outside of $S$ the limit can be seen to be $0$ or $\infty$ respectively. It hence makes sense to define the Onsager-Machlup functional for $\nu$ on $\Lambda\setminus\partial S$ as
\begin{eqnarray*}
K(a) = 
\begin{cases}
-\log\rho(a) & a \in \intr(S)\\
\infty & a \notin S.
\end{cases}
\end{eqnarray*}
For $a \in \partial S$, we define $K(a)$ to be the limit of $K$ from the interior:
\begin{eqnarray*}
K(a) = -\lim_{\substack{b\rightarrow a\\b\in\intr(S)}} \log\rho(b)\;\;\;a \in \partial S
\end{eqnarray*}
which is well defined due to the continuity of $\rho$ on $\intr(S)$. $K$ is then continuous on the whole of $S$.

\begin{remark}
If we were to define $K$ on $\partial S$ in the same way that we defined it on $\Lambda\setminus\partial S$, $K$ would have a positive jump at the boundary related to the geometry of $S$. This would mean that $K$ was not lower semi-continuous on $S$ which would cause problems when seeking minimisers. The definition we have chosen is appropriate: if any minimising sequence $(a_n)_{n\geq 1}\subseteq \intr(S)$ of $K$ has an accumulation point on $\partial S$, then $\nu$ has a mode at that point.
\end{remark}

If we have no prior knowledge about the interfaces and $\Lambda$ is compact, we could place a uniform prior on the whole of $\Lambda$. Otherwise we could either choose a prior with smaller support, or one that weights certain areas more than others.

\subsection{Priors on $X\times\Lambda$}
\label{ssec:prior}
We assume that the priors on the fields and the geometric parameters are independent, so that we may take the product measure $\mu_0\times\nu_0$ as our prior on $X\times\Lambda$. Note that if $F:X\times\Lambda\rightarrow L^\infty(D)$ denotes the construction map $(u,a)\mapsto u^a$ defined earlier by (\ref{eq:construction}), then our prior permeability/conductivity distribution on $L^\infty(D)$ is given by the pushforward\footnote{\mmd{Given a measurable map $F:(X,\mathcal{X})\rightarrow (Y,\mathcal{Y})$ between two measurable spaces, the pushforward of a measure $\mu$ on $X$ is the measure $F^\#\mu$ on $Y$ defined by $(F^\#\mu)(A) = \mu(F^{-1}(A))$ for $A \in \mathcal{Y}$. If a random variable $u$ on $X$ has law $\mu$, then the random variable $F(u)$ on $Y$ has law $F^\#\mu$.}} $\mu_0^* = F^\#(\mu_0\times\nu_0)$. This is much more cumbersome to deal with however, since for example $L^\infty(D)$ is not separable. It is for this reason we incorporate the mapping $F$ into the forward map $\mathcal{G}$. Assuming now that the prior $\mu_0\times\nu_0$ is as described above, we can define the Onsager-Machlup functional for measures $\mu$ on $X\times \Lambda$ which are equivalent to $\mu_0\times\nu_0$.

\begin{definition}[Onsager-Machlup functional II]
\label{def:om2}
Let $\mu$ be a measure on $X\times\Lambda$ equivalent to $\mu_0\times\nu_0$, where $\mu_0$ and $\nu_0$ satisfy the assumptions detailed above. Let $B^\delta(u,a)$ denote the ball of radius $\delta$ centred at $(u,a) \in X\times\Lambda$. A functional $I:X\times\Lambda\rightarrow\overline{\mathbb{R}}$ is called the \emph{Onsager-Machlup functional} for $\mu$ if,
\begin{enumerate}[(i)]
\item for each $(u,a),(v,b) \in E\times\intr(S)$,
\begin{eqnarray*}
\lim_{\delta\downarrow 0}\frac{\mu(B^\delta(u,a))}{\mu(B^\delta(v,b))} = \exp\left(I(v,b)-I(u,a)\right);
\end{eqnarray*}
\item for each $(u,a) \in E\times\partial S$,
\begin{eqnarray*}
I(u,a) = \lim_{\substack{b\rightarrow a\\b\in\intr(S)}} I(u,b);
\end{eqnarray*}
\item $I(u,a) = \infty$ for $u \notin E$ or $a \notin S$.
\end{enumerate}
\end{definition}

\section{Likelihood and Posterior Distribution}
\label{sec:post}
We return to the abstract setting mentioned in the introduction. Let $X$ be a separable Banach space, $\Lambda \subseteq \mathbb{R}^k$ and $Y = \mathbb{R}^J$. Suppose we have a forward operator $\mathcal{G}:X\times\Lambda\rightarrow Y$. If $(u,a)$ denotes the true input to our forward problem, we observe data $y \in Y$ given by
\begin{eqnarray*}
y = \mathcal{G}(u,a) + \eta
\end{eqnarray*} 
where $\eta \sim \mathbb{Q}_0 := N(0,\Gamma)$, $\Gamma\in\mathbb{R}^{J\times J}$ positive definite, is Gaussian noise on $Y$ independent of the prior. 

It is clear that we have $y|(u,a) \sim \mathbb{Q}_{u,a} := N(\mathcal{G}(u,a),\Gamma)$. We can use this to formally find the distribution of $(u,a)|y$. First note that
\begin{eqnarray*}
\mathbb{Q}_{u,a}(\dee y) = \exp\left(-\Phi(u,a;y) + \frac{1}{2}|y|_\Gamma^2\right)\mathbb{Q}_0(\dee y)
\end{eqnarray*}
where the potential (or negative log-likelihood) $\Phi:X\times\Lambda\times Y\rightarrow\mathbb{R}$ is given by
\begin{eqnarray}
\label{eq:potential}
\Phi(u,a;y) = \frac{1}{2}|\mathcal{G}(u,a) - y|_{\Gamma}^2.
\end{eqnarray}
Hence under suitable regularity conditions, Bayes' theorem tells us that the distribution $\mu$ of $(u,a)|y$ satisfies
\begin{eqnarray*}
\mu(\dee u, \dee a) \propto \exp\big(-\Phi(u,a;y)\big)\mu_0(\dee u)\nu_0(\dee a)
\end{eqnarray*}
after absorbing the $\exp\big(\frac{1}{2}|y|_\Gamma^2\big)$ term into the normalisation constant.

We now make this statement rigorous. To keep the situation general, we do not insist that $\Phi$ takes the form (\ref{eq:potential}), and instead assert only that $\Phi$ satisfies the following assumptions.

\begin{comment}
\begin{assumptions}
Let $X$ be a Banach space and $\Lambda \subseteq \mathbb{R}^k$. The function $\Phi:X\times \Lambda\times\mathbb{R}^J\rightarrow \mathbb{R}$ satisfies the following conditions.
\begin{enumerate}[(i)]
\item For every $\varepsilon > 0$ there is an $M = M(\varepsilon)\in \mathbb{R}$ such that for all $u \in X$ and all $a \in \Lambda$,
\begin{eqnarray*}
\Phi(u,a) \geq M - \varepsilon\|u\|^2_{X}
\end{eqnarray*}
\item $\Phi$ is locally bounded from above, i.e. for every $r > 0$ there exists $K = K(r) > 0$ such that, for all $u \in X$ with $\|u\|_{X} < r$ and all $a \in \Lambda$, we have
\begin{eqnarray*}
\Phi(u,a) \leq K
\end{eqnarray*}
\item $\Phi$ is locally Lipschitz continuous in its first component, i.e. for every $r > 0$ there exists $L = L(r) > 0$ such that for all $u,v \in X$ with $\|u\|_{X}, \|v\|_{X} < r$ and all $a \in \Lambda$, we have
\begin{eqnarray*}
|\Phi(u,a) - \Phi(v,a)| \leq L\|u-v\|_{X}
\end{eqnarray*}
\item $\Phi$ is continuous in its second component.
\end{enumerate}
\end{assumptions}
\end{comment}

\begin{assumptions}
\label{assump:map}
There exists $X'\times\Lambda' \subseteq X\times\Lambda$ such that
\begin{enumerate}[(i)]
\item for every $\eps > 0$ there is an $M_1(\eps) \in \mathbb{R}$ such that for all $u \in X'$ and all $a \in \Lambda'$
\begin{eqnarray*}
\Phi(u,a;y) \geq M_1(\eps) - \eps\|u\|_X^2;
\end{eqnarray*}
\item for each $u \in X'$ and $y \in Y$, the potential $\Phi(u,\cdot;y):\Lambda'\rightarrow\mathbb{R}$ is continuous;
\item there exists a strictly positive $M_2:\mathbb{R}^+\times\mathbb{R}^+\times\mathbb{R}^+\rightarrow\mathbb{R}^+$ monotonic non-decreasing separately in each argument, such that for each $r>0$, $u \in X'$ and $a \in \Lambda'$, and $y_1,y_2 \in Y$ with $|y_1|,|y_2| < r$,
\begin{eqnarray*}
|\Phi(u,a;y_1)  -\Phi(u,a;y_2)| \leq M_2(r,\|u\|_X,|a|)|y_1-y_2|;
\end{eqnarray*}
\item there exists a strictly positive $M_3:\mathbb{R}^+\times \Lambda\times Y\rightarrow\mathbb{R}^+$, continuous in its second component, such that  for each $r > 0$, $a \in \Lambda'$ and $y \in Y$, and $u_1,u_2 \in X'$ with $\|u_1\|_X, \|u_2\|_X < r$,
\begin{eqnarray*}
|\Phi(u_1,a;y) - \Phi(u_2,a;y)| \leq M_3(r,a,y)\|u_1-u_2\|_X.
\end{eqnarray*}
\end{enumerate}
\end{assumptions}

These assumptions are used in the proof of existence and well-posedness of the posterior distribution, which is given in the appendix:

\begin{theorem}[Existence and well-posedness]
\label{thm:post_exist}
Let Assumptions \ref{assump:map} hold. Assume that $(\mu_0\times\nu_0)(X'\times\Lambda') = 1$, and that $(\mu_0\times\nu_0)((X'\times\Lambda')\cap B) > 0$ for some bounded set $B \subseteq X\times\Lambda$. Then
\begin{enumerate}[(i)]
\item $\Phi$ is $\mu_0\times\nu_0\times\mathbb{Q}_0$-measurable;
\item for each $y \in Y$, $Z(y)$ given by
\begin{eqnarray*}
Z(y) = \int_{X\times\Lambda} \exp(-\Phi(u,a;y))\,\mu_0(\dee u)\nu_0(\dee a)
\end{eqnarray*}
is positive and finite, and so the probability measure $\mu^y$,
\begin{eqnarray}
\label{posterior}
\mu^y(\dee u,\dee a) = \frac{1}{Z(y)}\exp(-\Phi(u,a;y))\,\mu_0(\dee u)\nu_0(\dee a)
\end{eqnarray}
is well-defined.
\item Assume additionally that, for every fixed $r>0$, there exists $\eps > 0$ with
\begin{eqnarray*}
\exp(\eps\|u\|_X^2)(1 + M_2(r,\|u\|_X,|a|)^2) \in L^1_{\mu_0\times\nu_0}(X\times\Lambda;\mathbb{R}).
\end{eqnarray*}
Then there is $C(r) > 0$ such that for all $y,y' \in Y$ with $|y|,|y'| < r$, 
\begin{eqnarray*}
d_{\mathrm{Hell}}(\mu^y,\mu^{y'}) \leq C|y-y'|.
\end{eqnarray*}
\end{enumerate} 
\end{theorem}

\mmd{\begin{remark}
In this paper we are focused on the case when the field prior $\mu_0$ is taken to be Gaussian. However, the above existence and well-posedness result still holds if, for example, $\mu_0$ is taken to be Besov rather than Gaussian, since a Fernique-type theorem holds for such priors \cite{besovprior,lecturenotes}.
\end{remark}}

We show that for both choices of test models, the potential (\ref{eq:potential}) satisfies Assumptions \ref{assump:map}:
\begin{proposition}
Let $X = C^0(D;\mathbb{R}^N)$, and let $\mathcal{G}:X\times\Lambda\rightarrow Y$ denote the forward map corresponding to either the groundwater flow or EIT problem, as detailed in section \ref{sec:fwd}. Let $y \in Y$ and let $\Gamma \in \mathbb{R}^{J\times J}$ be positive definite. Define the potential $\Phi:X\times\Lambda\times Y\rightarrow\mathbb{R}$ by
\begin{eqnarray*}
\Phi(u,a;y) = \frac{1}{2}|\mathcal{G}(u,a) - y|^2_\Gamma.
\end{eqnarray*}
Then $\Phi$ satisfies Assumptions \ref{assump:map}, with $X'\times\Lambda' = X\times\Lambda$.
\end{proposition}

\begin{proof}
\begin{enumerate}[(i)]
\item $\Phi \geq 0$ so this is true with $M_1 \equiv 0$.
\item Fix $u \in X'$ and $y \in Y$. Propositions \ref{cor:cont_gwf} and \ref{prop:cont_eit} tell us that $\mathcal{G}(u,\cdot)$ is continuous for either choice of test model. The map $z\mapsto|z-y|^2_{\Gamma}$ is continuous, and so $\Phi(u,\cdot;y)$ is continuous too. 
\item A consequence of Propositions \ref{cor:cont_gwf} and \ref{prop:cont_eit} is that for each $u \in X$ and $a \in \Lambda$, $\mathcal{G}(u,a)$ can be bounded in terms of $\|u\|_X$ and $|a|$. The result then follows from the local Lipschitz property of the map $y\mapsto|y|^2$.
\item Propositions \ref{cor:cont_gwf} and \ref{prop:cont_eit} tell us that $\mathcal{G}(\cdot,a)$ is locally Lipschitz for either choice of test model. The map $z\mapsto|z-y|^2_{\Gamma}$ is locally Lipschitz, and hence we conclude that $\Phi(\cdot,a;y)$ is locally Lipschitz, with Lipschitz constant independent of $a$. 
\end{enumerate}
\end{proof}

\mmd{With a choice of prior as described in section \ref{sec:prior}, we can therefore apply Theorem \ref{thm:post_exist} in the cases where the forward map is one of the two described in section \ref{sec:fwd} and the observational noise is Gaussian. In this case, the constant $M_2(r,\|u\|_X,|a|)$ appearing in Assumptions \ref{assump:map}(iii) is independent of $\|u\|_X$ and $|a|$, and so the integrability condition (iii) in Theorem \ref{thm:post_exist} always holds via Fernique's theorem. The condition on positivity of a bounded set can be seen by taking, for example, $B = B^1(0)\times S$, where $S$ is the (compact) support of $\nu_0$.}

\section{MAP Estimators}
\label{sec:map}

In subsection \ref{ssec:map} we characterise the MAP estimators for the posterior $\mu$ in terms of the Onsager-Machlup functional for $\mu$. In subsection \ref{ssec:fomin} we relate this Onsager-Machlup functional to the Fomin derivative of $\mu$, with reference to the work \cite{wmap}.

\subsection{MAP Estimators and the Onsager-Machlup Functional}
\label{ssec:map}
Throughout this section we assume that $\mu$ is given by (\ref{posterior}). Furthermore we assume that $\mu_0$ has mean zero for simplicity. Additionally, when we assume that Assumptions \ref{assump:map} hold, we will assume that $X'\times\Lambda' = X\times\Lambda$.

Suppressing the dependence of $\Phi$ on the data $y$ since it is not relevant in the sequel, we define the functional $I:X\times\Lambda\rightarrow\mathbb{R}$ by
\begin{eqnarray}
\label{om}
I(u,a) = \Phi(u,a) + J(u) + K(a)
\end{eqnarray}
where $J, K$ are as defined in subsections \ref{ssec:priorfield}, \ref{ssec:priorgeo} respectively. In this section we attain the following three results concerning $I$ and $\mu$, which are proved in the appendix.

\begin{theorem}
\label{omthm}
Let Assumptions \ref{assump:map} hold. Then the function $I$ defined by (\ref{om}) is the Onsager-Machlup functional for $\mu$, where the Onsager-Machlup functional is as defined in Definition \ref{def:om2}.
\end{theorem}

\begin{theorem}
\label{exist}
Let Assumptions \ref{assump:map} hold. Then there exists $(\bar{u},\bar{a}) \in E\times S$ such that
\begin{eqnarray*}
I(\bar{u},\bar{a}) = \inf\{I(u,a)\,|\, u \in E,\;a \in S\}.
\end{eqnarray*}
Furthermore, if $(u_n,a_n)_{n\geq 1}$ is a minimising sequence satisfying $I(u_n,a_n)\rightarrow I(\bar{u},\bar{a})$, then there is a subsequence $(u_{n_k},a_{n_k})_{k\geq 1}$ converging to $(\bar{u},\bar{a})$ (strongly) in $E\times S$.
\end{theorem}

\begin{theorem}
\label{thm:lim_is_map}
Let Assumptions \ref{assump:map} hold. Assume also that there exists an $M \in \mathbb{R}$ such that $\Phi(u,a) \geq M$ for any $(u,a) \in X\times\Lambda$.
\begin{enumerate}[(i)]
\item Let $(u^\delta,a^\delta) = \underset{(u,a)\in X\times\Lambda}{\mathrm{argmax}} \mu(B^\delta(u,a))$. There is a $(\bar{u},\bar{a}) \in E\times S$ and a subsequence of $(u^\delta,a^\delta)_{\delta > 0}$ which converges to $(\bar{u},\bar{a})$ strongly in $X\times\Lambda$.
\item The limit $(\bar{u},\bar{a})$ is a MAP estimator and minimiser of $I$.
\end{enumerate}
\end{theorem}

A consequence of Theorem $\ref{thm:lim_is_map}$ is that, under its assumptions, MAP estimators and minimisers of the Onsager-Machlup functional are equivalent. The proof of this corollary is identical to that of Corollary 3.10 in \cite{map}:

\begin{corollary}
Under the conditions of Theorem \ref{thm:lim_is_map} we have the following.
\begin{enumerate}[(i)]
\item Any MAP estimator minimises the Onsager-Machlup functional $I$
\item Any $(u^*,a^*)\in E\times S$ which minimises the Onsager-Machlup functional $I$ is a MAP estimator for the measure $\mu$ given by (\ref{posterior}).
\end{enumerate}
\end{corollary}

\subsection{The Fomin Derivative Approach}
\label{ssec:fomin}
In recent work of Helin and Burger \cite{wmap}, the concept of MAP estimators was generalised to weak MAP (wMAP) estimators using the notion of Fomin differentiability of measures. The definition of wMAP estimators is such that if $\hat{u}$ is a MAP estimator then it is a wMAP estimator, but not necessarily vice versa. Under certain assumptions, they show that wMAP estimators are equivalent to minimisers of a particular functional. The assumptions do not hold in our case, since our forward map is non-linear and our prior $\mu_0\times\nu_0$ isn't necessarily convex, however the functional agrees with our objective functional $I$. Thus in what follows we provide a link between the Fomin derivative of the posterior $\mu$ and our objective functional $I$. 

The Fomin derivative of a measure on a Banach space $X$ equipped with its Borel $\sigma$-algebra $\mathcal{B}(X)$ is defined as follows.
\begin{definition}
A measure $\lambda$ on $X$ is called Fomin differentiable along the vector $z \in X$ if, for every set $A \in \mathcal{B}(X)$, there exists a finite limit
\begin{eqnarray*}
d_z\lambda(A) = \lim_{t\rightarrow 0}\frac{\lambda(A+tz) - \lambda(A)}{t}.
\end{eqnarray*}
The Radon-Nikodym density of $d_z\lambda$ with respect to $\lambda$ is denoted $\beta^\lambda_z$, and is called the logarithmic derivative of $\lambda$ along $z$.
\end{definition}
\begin{example}
\label{ex:fomin}
\begin{enumerate}[(i)]
\item Let $\nu_0$ be a measure on $\mathbb{R}^k$ with Lebesgue density $\rho$, supported and continuously differentiable on $S\subseteq\mathbb{R}^k$. Then for any $a \in\intr(S)$ and $b \in \mathbb{R}^k$ we have
\begin{eqnarray*}
\beta^{\nu_0}_b(a) = \frac{\nabla\rho(a)}{\rho(a)}\cdot b = \partial_{b}\log\rho(a).
\end{eqnarray*}
\item Let $\mu_0$ be a Gaussian measure on a Banach space $X$ with Cameron-Martin space $(E,\langle\cdot,\cdot\rangle_E)$. Then for any $u \in X$ and $h \in E$ we have
\begin{eqnarray*}
\beta^{\mu_0}_h(u) = -\langle u,h\rangle_E.
\end{eqnarray*}
This follows from the Cameron-Martin and dominated convergence theorems.
\item Again using the Cameron-Martin and dominated convergence theorems, we see that with $\nu_0$ and $\mu_0$ as above, for any $(u,a) \in X\times\intr(S)$ and $(h,b) \in E\times \mathbb{R}^k$,
\begin{eqnarray*}
\beta^{\mu_0\times\nu_0}_{(h,b)} (u,a) = \beta^{\mu_0}_h + \beta^{\nu_0}_b.
\end{eqnarray*}

We can use the above example to characterise the Fomin derivative of our posterior distribution $\mu$, given by (\ref{posterior}).
\begin{theorem}
Assume that $\Phi:X\times\Lambda\rightarrow\mathbb{R}$ is bounded measurable with uniformly bounded derivative, and assume that $\rho$ is continuously differentiable on $S$. Then for each $(u,a) \in X\times\intr(S)$ and $(h,b) \in E\times\mathbb{R}^k$, we have
\begin{eqnarray*}
\beta^{\mu}_{(h,b)}(u,a) &= -\partial_{(h,b)}\Phi(u,a) - \langle u,h\rangle_E + \partial_b\log\rho(a)\\
&= -\partial_{(h,b)} I(u,a)
\end{eqnarray*}
Therefore, $(\hat{u},\hat{a})$ is a critical point of $I$ if and only if $\beta^\mu_{(h,b)}(\hat{u},\hat{a}) = 0$ for all $(h,b) \in E\times\mathbb{R}^k$. 
\end{theorem}

\begin{proof}
We use result $(2.1.13)$ from \cite{bogachev-fomin}, which tells us that if $\lambda$ is a measure differentiable along $z$ and $f$ is a bounded measurable function with uniformly bounded partial derivative $\partial_zf$, then the measure $f\cdot\lambda$ is differentiable along $h$ as well and
\begin{eqnarray*}
d_z(f\cdot\lambda) = \partial_z f\cdot \lambda + f\cdot d_z\lambda.
\end{eqnarray*}
We apply this result with $\lambda = \mu_0\times\nu_0$, $f = \exp(-\Phi)/Z$ and $z = (u,a)$. Note that $f$ satisfies the assumptions of (2.1.13) due to the assumptions on $\Phi$. The result then follows using Example \ref{ex:fomin} (iii) above.
\end{proof}

\end{enumerate}
\end{example}

\section{Numerical Experiments}
\label{sec:num}
In this section we perform some numerical experiments related to the theory above for a variety of geometric models, in the case of the groundwater flow forward map \mmd{introduced in subsection \ref{ssec:gwf}}. We both compute minimisers of the relevant Onsager-Machlup functional (i.e. MAP estimators), and we sample the posterior distribution using a state-of-the-art function space Metropolis-Hastings MCMC method. We then relate the samples to the MAP estimators. From these numerical experiments we observe the following behaviour of the posterior distribution.
\begin{enumerate}
\item The posterior distribution can be highly multi-modal, especially when the \mmd{parameterised} geometry is non-trivial. This is evident from the sensitivity of the minimisation of the objective functional on its initial state, and the behaviour of MCMC chains initiallised at these calculated minimisers.
\item When the geometry is incorrect the fields attempt to compensate, which presumably contributes to the existence of multiple local minimisers of the objective functional; this occurs in both the MAP estimation and the MCMC samples. A consequence is that many of the local minimisers lack the desired sharp interfaces.  These minimisers could however be used to suggest more appropriate geometric parameters for the initialisation. 
\item The mixing rates of MCMC chains have a strong dependence upon which local minimiser they are initialised at: acceptance rates can vary wildly when the initial state is changed but all other parameters are kept fixed. \mmd{This provides some insight into the shape of the posterior distribution.}
\item Though often there are many local minimisers, they can be separated into classes of minimisers sharing similar characteristics, such as close geometry. MCMC chains typically tend to stay within these classes, which can be observed by monitoring the closest local minimiser to an MCMC chain's state at each step. \mmd{This suggests that the posterior can possess several clusters of nearby modes.}
\end{enumerate}
\mmd{One conclusion we can draw from the above points is that there are often many different geometries that are consistent with the data. This is not necessarily an effect of noise on the measurements, and the effect may persist as the noise level goes to zero, since it is unknown if these geometric parameters are uniquely identifiable in general.}

\subsection{Test Models}
We consider three different geometric models: a two parameter, two layer model; a five parameter, three layer model with fault; and a five parameter channelised model.

In what follows, as in Example \ref{ex:laplacian}, we define the negative Laplacian with Neumann boundary conditions:
\begin{eqnarray*}
A = -\Delta,\;\;\;\mathcal{D}(A) &= \left\{u \in H^2(D)\;\bigg|\;\frac{\dee u}{\dee \nu} = 0\text{ on }\partial D, \int_Du(x)\,\dee x = 0\right\}.
\end{eqnarray*}
Recall that if $u \sim N(0,A^{-\alpha})$ with $\alpha > d/2$, then $u$ is almost surely continuous \cite{lecturenotes}.

\subsubsection{Model 1 (Two layer)}
This model is described in Example \ref{ex:straight}. The geometric parameters $a = (a^1,a^2)$ are defined as in Figure \ref{fig:model1}. For simulations, we use the choice of prior
\begin{eqnarray*}
\mu_0 &= N(1,A^{-1.4})\times N(-1,A^{-1.8}),\\
\nu_0 &= U([0,1])\times U([0,1]).
\end{eqnarray*}

\begin{figure}
\begin{center}
\frame{\includegraphics[width=0.4\textwidth]{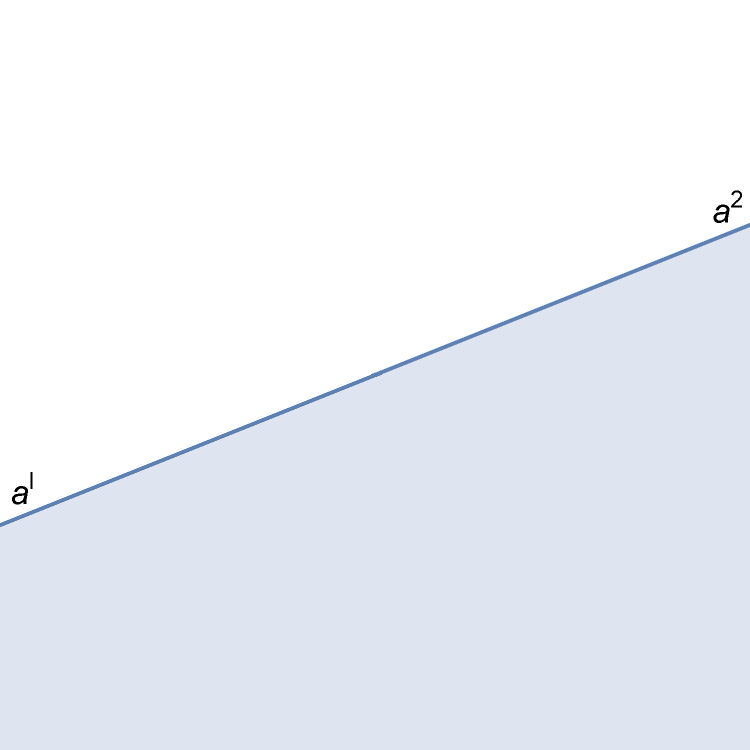}}
\caption{The definition of the geometric parameters $a = (a^1,a^2)$ in Model 1.}
\label{fig:model1}
\end{center}
\end{figure}

\subsubsection{Model 2 (Three layer with fault)}
This model is described in \cite{kui}, where it is labelled \emph{Test Model 1}. The geometric parameters $a = (a^1,a^2,a^3,a^4,a^5)$ are defined as in Figure \ref{fig:model2}, with the fault occurring at $x = 0.55$. For simulations, we use the choice of prior
\begin{eqnarray*}
\mu_0 &= N(2,2A^{-1.4})\times N(0,A^{-1.8})\times N(-2,2A^{-1.4}),\\
\nu_0 &= U(S)\times U(S)\times U([-0.3,0.3]),
\end{eqnarray*}
where $S\subseteq [0,1]^2$ is the simplex $S = \{(x,y)\;|\; 0\leq x \leq 1, x \leq y \leq 1\}$.

\begin{figure}
\begin{center}
\frame{\includegraphics[width=0.4\textwidth]{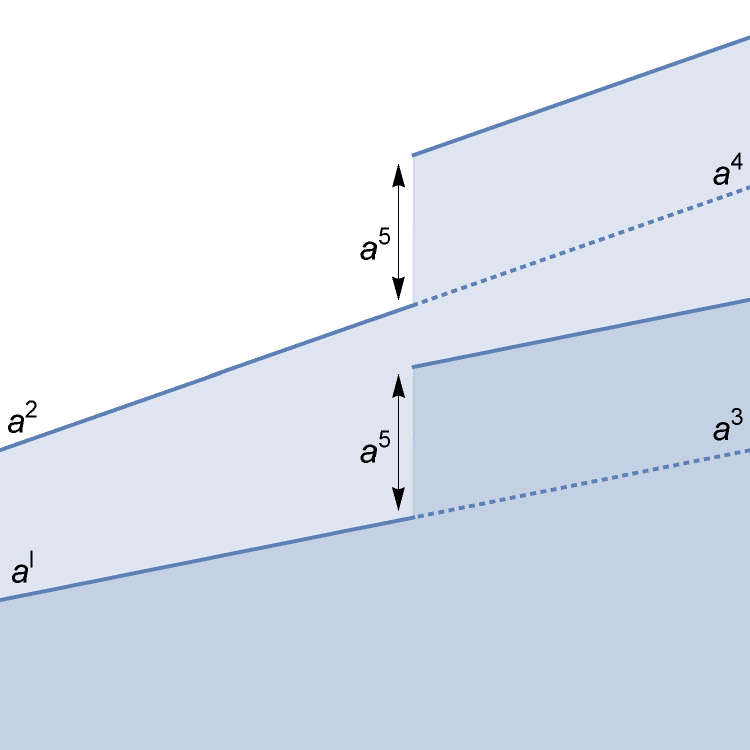}}
\caption{The definition of the geometric parameters $a = (a^1,a^2,a^3,a^4,a^5)$ in Model 2.}
\label{fig:model2}
\end{center}
\end{figure}

\subsubsection{Model 3 (Channel)}
This model is described in \cite{kui}, where it is labelled \emph{Test Model 2}. The geometric parameters $a = (a^1,a^2,a^3,a^4,a^5)$ are defined as in Figure \ref{fig:model3}. Here $a^1,a^2,a^3,a^4,a^5$ represent the channel amplitude, frequency, angle, initial point and width respectively. For simulations, we use the choice of prior
\begin{eqnarray*}
\mu_0 &= N(1,A^{-1.4})\times N(-1,A^{-1.8}),\\
\nu_0 &= U([0,1])\times U([\pi,4\pi])\times U([-\pi/4,\pi/4])\times U([0,1])\times U([0,0.4]).
\end{eqnarray*}

\begin{figure}
\begin{center}
\frame{\includegraphics[width=0.4\textwidth]{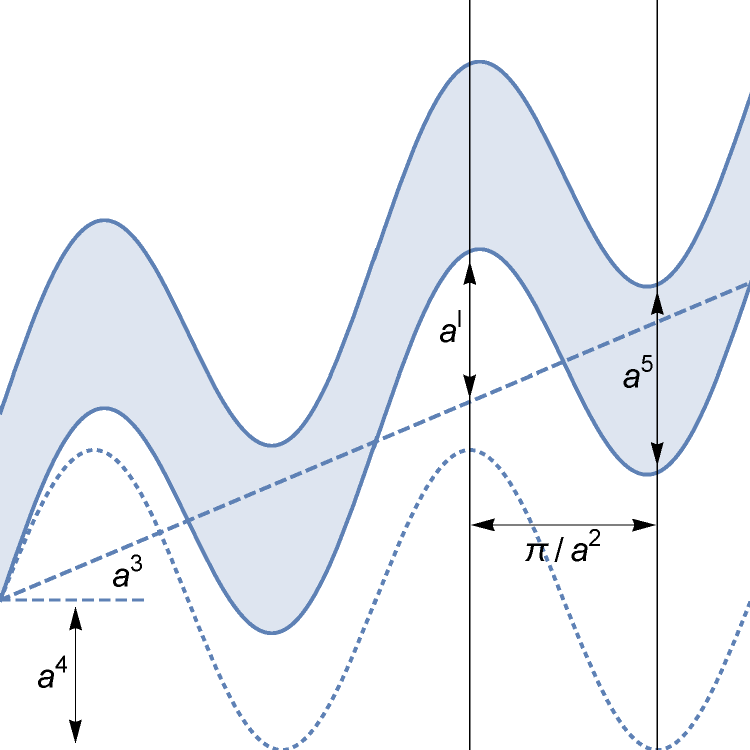}}
\caption{The definition of the geometric parameters $a = (a^1,a^2,a^3,a^4,a^5)$ in Model 3.}
\label{fig:model3}
\end{center}
\end{figure}

For each model, we fix a true permeability $(u^\dagger,a^\dagger)$ as a draw from the corresponding prior distribution, generated on a mesh of $256^2$ points. For the forward model, we take the coefficient map $\sigma(\cdot) = \exp(\cdot)$. We observe the pressure on a grid $(x_i)_{i=1}^{25}$ of 25 uniformly spaced points, via the maps (\ref{eq:obs_op}) with $\eps = 0.05$. We add i.i.d. Gaussian noise $N(0,\gamma^2)$ to each observation, taking $\gamma = 0.01$. The resulting relative errors on the data can be seen in Table \ref{tab:errors}. Small relative errors of this size typically make the posterior distribution hard to sample as they lead to measure concentration phenomena; MAP estimation can thus be particularly important.

\begin{table}[h!]
\begin{center}
\begin{tabular}{ccc}
\br
\textbf{Model Number} & \textbf{Mean relative error (\%)} & \textbf{Range of relative errors (\%)}\\
\mr
$1$ & $0.5$ & $0.02-3.5$\\
$2$ & $0.9$ & $0.1-4.0$\\
$3$ & $0.3$ & $0.1-1.0$\\
\br
\end{tabular}
\end{center}
\caption{The relative error on the data, when each measurement is perturbed by an instance of $N(0,0.01^2)$ noise.}
\label{tab:errors}
\end{table}

\subsection{MAP Estimation}
\mmd{Based on the theory in section \ref{sec:map}, we can calculate MAP estimators by minimising the Onsager-Machlup functional for the posterior distribution.} We compute local minimisers of the Onsager-Machlup functional using the following iterative alternating method.
\begin{algorithm}
\begin{enumerate}[1.]
\item Choose an initial state $(u_0,a_0) \in X\times\Lambda$.
\item Update the geometric parameters simultaneously using the Nelder-Mead algorithm.
\item Update each field individually using a line-search in the direction provided by the Gauss-Newton algorithm.
\item Go to 2.
\end{enumerate}
\end{algorithm}
The Nelder-Mead and Gauss-Newton algorithms are discussed in \cite{nocedal}, in sections 9.5 and 10.3 respectively.
Since we do not update the fields and geometric parameters simultaneously, it is possible that this algorithm will get caught in a saddle point: consider for example the function $f:\mathbb{R}\times\mathbb{R}\rightarrow\mathbb{R}$, $f(x,y) = xy$, at the point $(0,0)$, being minimised alternately in the coordinate directions. Hence once the algorithm stalls, we propose a large number of random simultaneous updates in an attempt to find a lower functional value. If this is successful, we return to step 2 \mmd{of the algorithm}. We terminate the algorithm once the difference between successive values of $\Phi$ is below TOL = $10^{-5}$. Calculations are performed on a mesh of $64^2$ points in order to avoid an inverse crime.

\mmd{To ensure that we explore the support of the posterior distribution,} we choose a variety of initial states $(u_0,a_0) \mmd{\in X\times\Lambda}$ \mmd{for the minimisation} such that $I(u_0,a_0) < \infty$ \mmd{in the continuum setting}. To this end, we let $a_0$ be a draw from the prior distribution $\nu_0$, and take $u_0$ to lie in the Cameron-Martin space of $\mu_0$. Specifically, if a component of $u \in X$ has prior distribution $N(m,A^{-\alpha})$, we take the corresponding component of $u_0$ to be a draw from $N(m,A^{-\alpha-d/2})$. Output of the algorithm is shown in Figures \ref{maps1}-\ref{maps3}.

We first comment on the minimisers of the Onsager-Machlup functional for Model 1. Generally the geometric parameters are closely recovered \mmd{regardless of the initialisation state}, though there is more variation in the fields. In the simulations where the geometry is inaccurate, for example simulations 7, 17 and 46, the fields can be seen to be compensating \mmd{by forming a `soft' interface where the true interface is}. 

The minimisers associated with Model 2 admit much more variation, \mmd{though it is possible to partition them into smaller subsets of minimisers which share similar characteristics to one another}, as mentioned in point (iv) at the beginning of the section. \mmd{The clustering of the different minimisers is performed by eye, classifying them according to similar geometric parameters. Additionally we have an Other class, containing the minimisers which do not appear similar to one another nor appear to fit into any other class. We see later with MCMC simulations that these states do still act as local maximisers of the posterior probability.}

The minimisers of the Onsager-Machlup functional for Model 3 show even more variation than those for Model 2, with the geometry in half of the minimisers not even being close to the true geometry. In the cases where the geometry is drastically wrong the fields have again attempted to compensate. This behaviour is particularly evident in the Other class, and is echoed in the MCMC simulations later. \mmd{The Other class here is much larger than for Model 2, though as with Model 2 these states do appear to act as distinct local maximisers of the posterior probability}.

This multi-modality of the posterior distribution is not unexpected. The paper \cite{historymatching} considers the history matching problem in reservoir simulation, in which inference is done jointly on both geometric and permeability parameters in the IC fault model. Though the forward map and observation maps are different in our model, we observe the same clustering of nearby local MAP estimators, and increased multi-modality as the dimension of the parameter space increases. In \cite{historymatching} it is observed that the global minimum often does not correspond to the truth, especially in the presence of measurement noise, and so all local minimisers of the Onsager-Machlup functional should be sought before drawing conclusions about the permeabilities -- this appears to be the case in our model as well. We note that MCMC can be useful in identifying a range of such minimisers, in view of the links established in the next subsection between MCMC and MAP estimation.

\subsection{MCMC and Local Minimisers}
We now observe the behaviour of MCMC chains initialised at these local minimisers of the Onsager-Machlup functional. We use a Metropolis-within-Gibbs algorithm for the sampling, alternating between preconditioned Crank-Nicolson (pCN) updates for the fields, see \cite{pcn} for details, and Random Walk Metropolis updates for the geometric parameters. Again, simulations are performed on a mesh of $64^2$ points in order to avoid an inverse crime. $10^5$ samples are taken for each chain, with the initial $2\times 10^4$ discarded as burn-in. The conditional means calculated from the samples are shown in Figures \ref{cm1}-\ref{cm3}.

We monitor the value that $\Phi$ takes along the chain $(u^{(n)},a^{(n)})$, and compare it with the value $\Phi$ takes on the local minimisers $(u_{\mathrm{MAP}}^i,a_{\mathrm{MAP}}^i)$. This is shown in Figures \ref{phitrace1}-\ref{phitrace3}, with the horizontal lines being the different values of $\Phi(u_{\mathrm{MAP}}^i,a_{\mathrm{MAP}}^i)$. Note that it makes no sense to monitor the value that the objective functional $I$ takes along the chain as the fields almost surely do not lie in the corresponding Cameron-Martin spaces, and so $I$ is almost surely infinite along the chain \mmd{in the continuum setting}.

In addition, we monitor which minimiser the chain is nearest at each step, \mmd{in the permeability space}. Specifically, we look at
\begin{eqnarray}
\label{eq:mn}
m_n := \underset{i}{\mathrm{argmin}}\;\|F(u^{(n)},a^{(n)})-F(u^i_{\mathrm{MAP}},a_{\mathrm{MAP}}^i)\|_{L^2(D)}
\end{eqnarray}
where $F:X\times\Lambda\rightarrow L^\infty(D)$ is \mmd{the construction map (\ref{eq:construction})} from the state space to the permeability space. \mmd{We make the choice of the $L^2$ norm over the $L^\infty$ norm for the permeability space to avoid over-penalising incorrect geometry.} A selection of traces of $m_n$ are shown in Figures \ref{argmin1}-\ref{argmin3}. These illustrate that even though some of the local minimisers are very far from from the true log-permeability, they do indeed act as local maximisers of the posterior probability. \new{Moreover, they show the interaction between the different classes of minimisers in the cases of Models 2 and 3. Specifically, they show that the MCMC chains can easily move within these classes, but moving between classes is more difficult.}

\mmd{We now discuss the above monitored quantities, and their relation to MAP estimators, on a model-by-model basis.} Despite the slight variation in the fields of the minimisers from Model 1, the conditional means arising from the MCMC are nearly all identical. Simulation 23 stands out from the rest due to its slightly incorrect geometry -- this effect can be seen in the trace plot of $\Phi$, Figure \ref{phitrace1}, \mmd{where} the value of $\Phi$ remains larger than the simulations started elsewhere. \mmd{The traces of $\Phi$ for all other initialisations behave similarly to one another, taking similar misfit values after $2\times 10^4$ samples.} From Figure \ref{argmin1}, it can be seen that the MCMC chains \mmd{considered} all spend a lot of time close to MAP estimator 38, despite this not being the estimator with the lowest functional value.

For Model 2, typically the conditional means within the different classes are very similar to one another. Classes A and C resemble each other, and Class B has compensated for incorrect geometry with the centre field. Faults have developed in Class D, though there is still some compensation in the field. The centre field and a small fault has appeared in Class E, but again the fields are compensating. The geometric parameters for the permeabilities in the Other class remain relatively unchanged, but the fields have more freedom to attain a lower misfit than in the Onsager-Machlup functional minimisation due to the lack of regularisation term. Figure \ref{phitrace2} shows evidence for a number of local minima with a large data misfit value $\Phi$, with some chains appearing to remain stuck in their vicinity. The four chains visible in Figure \ref{phitrace2} (top) correspond to chains 49, 47, 45 and 43, from highest to lowest $\Phi$ value, all lying in the Other class -- \mmd{despite their significantly incorrect geometry, the corresponding MAP estimators appear to be genuine local maximisers of the posterior probability.}

In the channelised model, Model 3, there is yet more variation between local minimisers. Here the compensation effect by the fields is even more apparent \emph{in the conditional means}, especially in the Other class. From Figure \ref{phitrace3} it appears that the local minima are much sharper and more sparsely distributed than the previous two models. Again the chains with the largest $\Phi$ values were initialised at minimisers in the Other class, \mmd{suggesting the existence of many posterior modes with incorrect geometry.}

The mixing of the MCMC chains varies heavily based on the initialisation points of the chains: with the same jump parameters for the field and geometric parameter proposals, acceptance rates vary largely based on which minimiser the chain was started from. This indicates that some of the minima are much sharper than others. This is also evident from the traces of $m_n$ defined above, Figures \ref{argmin1}-\ref{argmin3}, especially in Model 3. Note also from these figures that the nearest local minimum typically lies in the same class as the initialisation state, though jumps between classes are possible. Though not shown, in Model 2, \mmd{whenever} the initial state lies in Class A, then the nearest minimiser always lies in Class A.

%%%%%%
% MAP ESTIMATORS

\begin{figure}[h!]
\begin{center}
\includegraphics[width=\textwidth,trim=2.7cm 3cm 3cm 3cm]{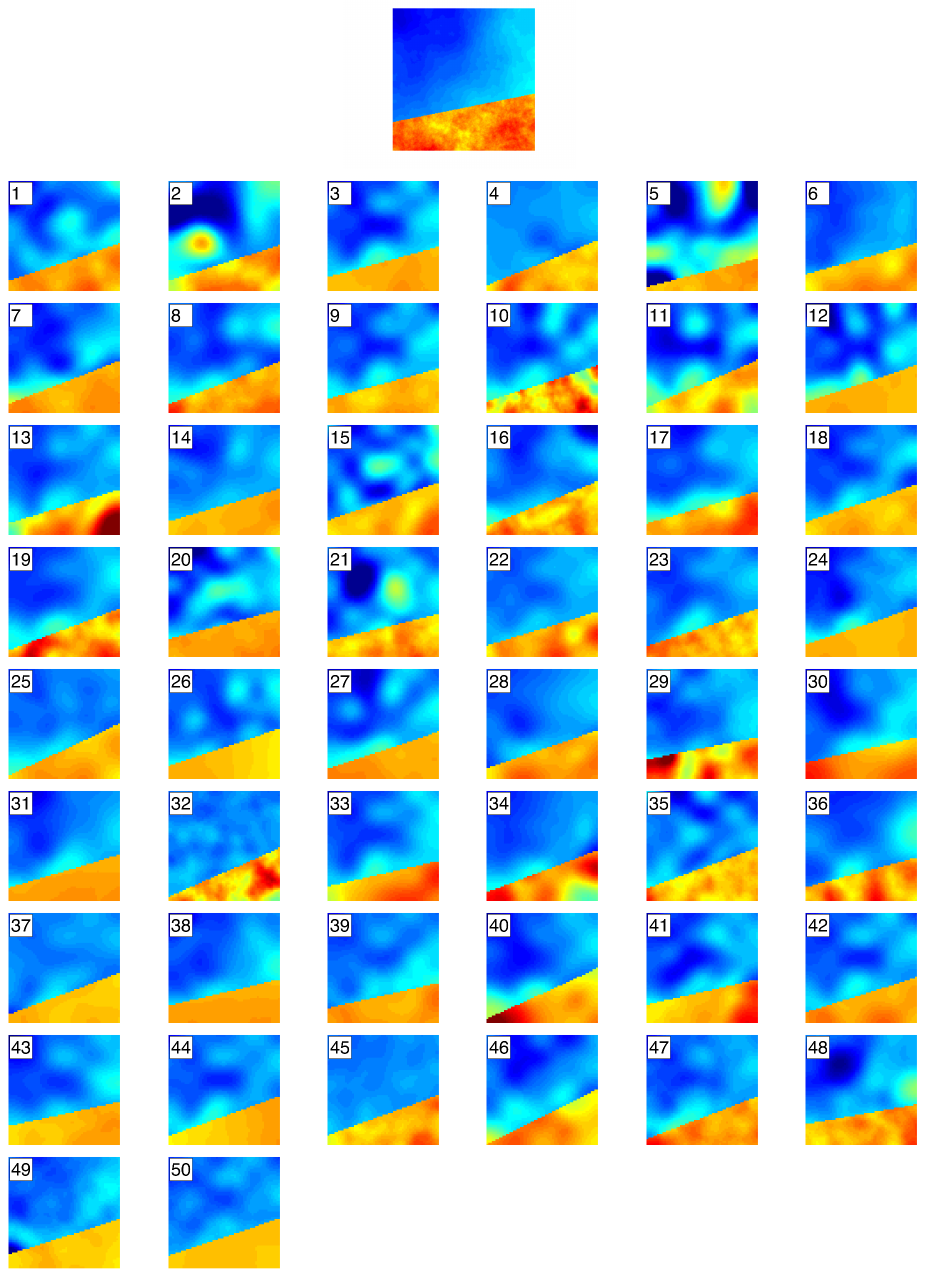}\\
\end{center}
\vspace{-0.5cm}
\caption{(Model 1) the true log-permeability field (top), and 50 local minimisers arising from minimisation initialised at draws from a smoothed prior distribution. Simulation 12 has the lowest functional value, with $I(u_{\mathrm{MAP}}^{12},a_{\mathrm{MAP}}^{12}) = 2847$.}
\label{maps1}
\end{figure}

\begin{figure}[h!]
\begin{center}
\includegraphics[width=\textwidth,trim=2.6cm 3cm 3cm 3cm]{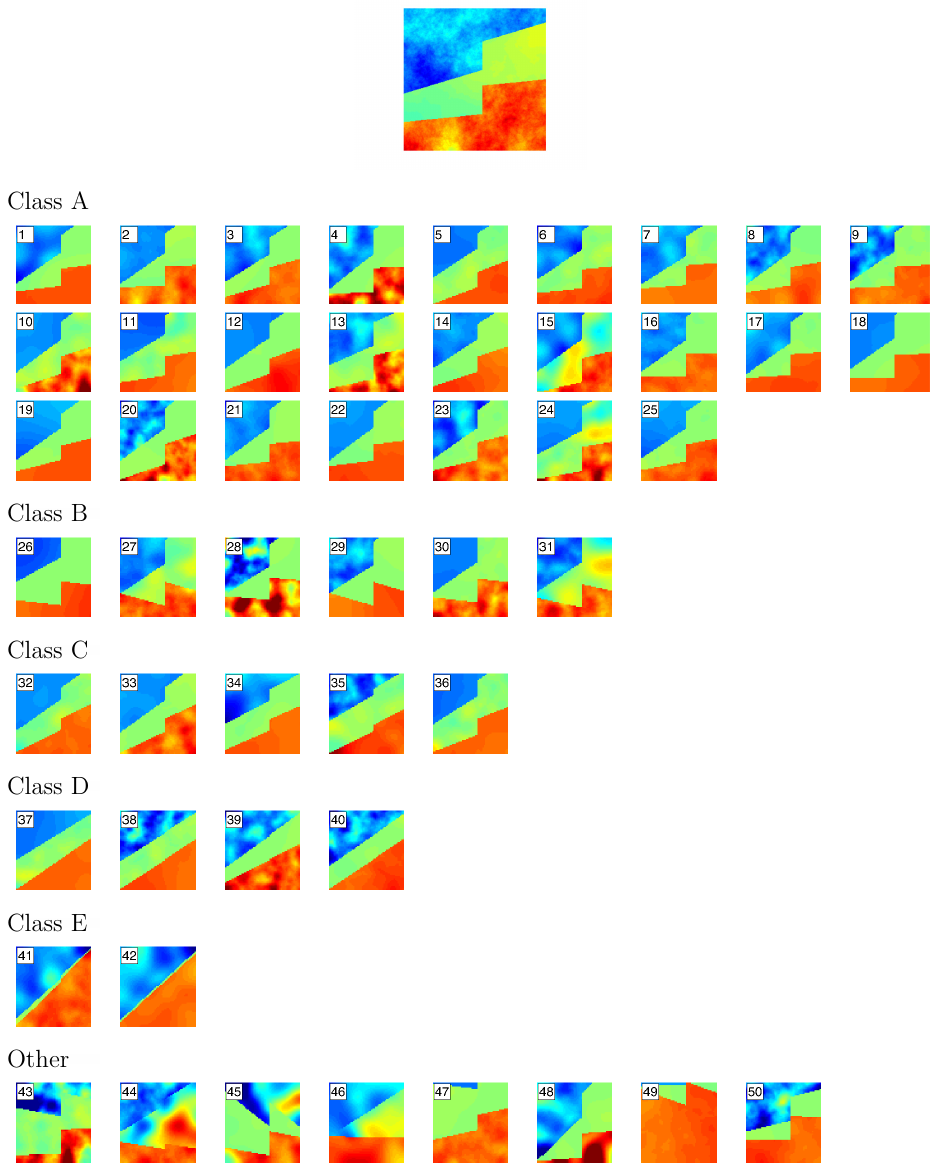}\\
\end{center}
\vspace{-2cm}
\caption{(Model 2) the true log-permeability field (top), and 50 local minimisers arising from minimisation initialised at draws from a smoothed prior distribution. Simulation 7 has the lowest functional value, with $I(u_{\mathrm{MAP}}^{7},a_{\mathrm{MAP}}^{7}) = 2567$. The minimisers have been divided into classes based on similar characteristics.}
\label{maps2}
\end{figure}

\begin{figure}[h!]
\begin{center}
\includegraphics[width=\textwidth,trim=2.6cm 3cm 3cm 3cm]{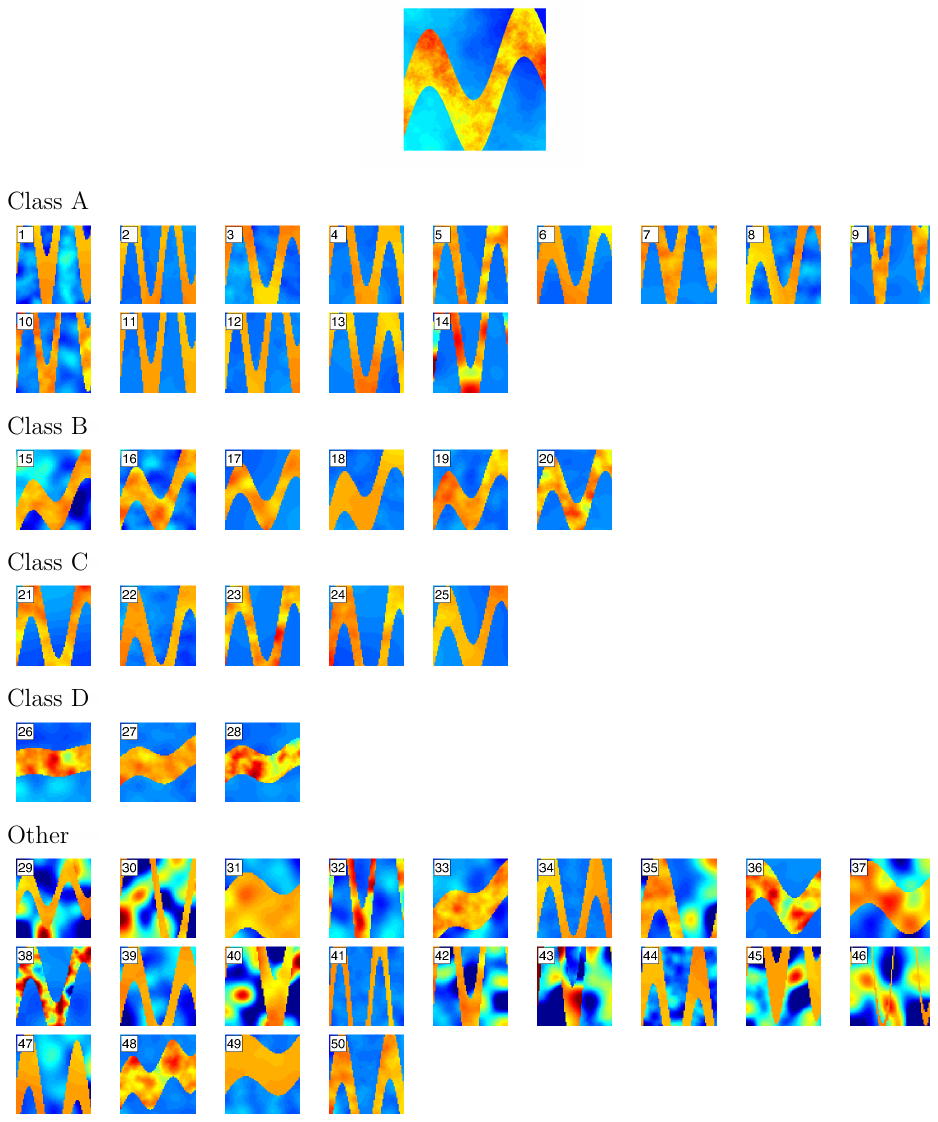}\\
\end{center}
\vspace{-2.7cm}
\caption{(Model 3) the true log-permeability field (top), and 50 local minimisers arising from minimisation initialised at draws from a smoothed prior distribution. Simulation 20 has the lowest functional value, with $I(u_{\mathrm{MAP}}^{20},a_{\mathrm{MAP}}^{20}) = 2117$. The minimisers have been divided into classes based on similar characteristics.}
\label{maps3}
\end{figure}

%%%%%%
% CONDITIONAL MEANS

\begin{figure}[h!]
\begin{center}
\includegraphics[width=\textwidth,trim=2.7cm 3cm 3cm 3cm]{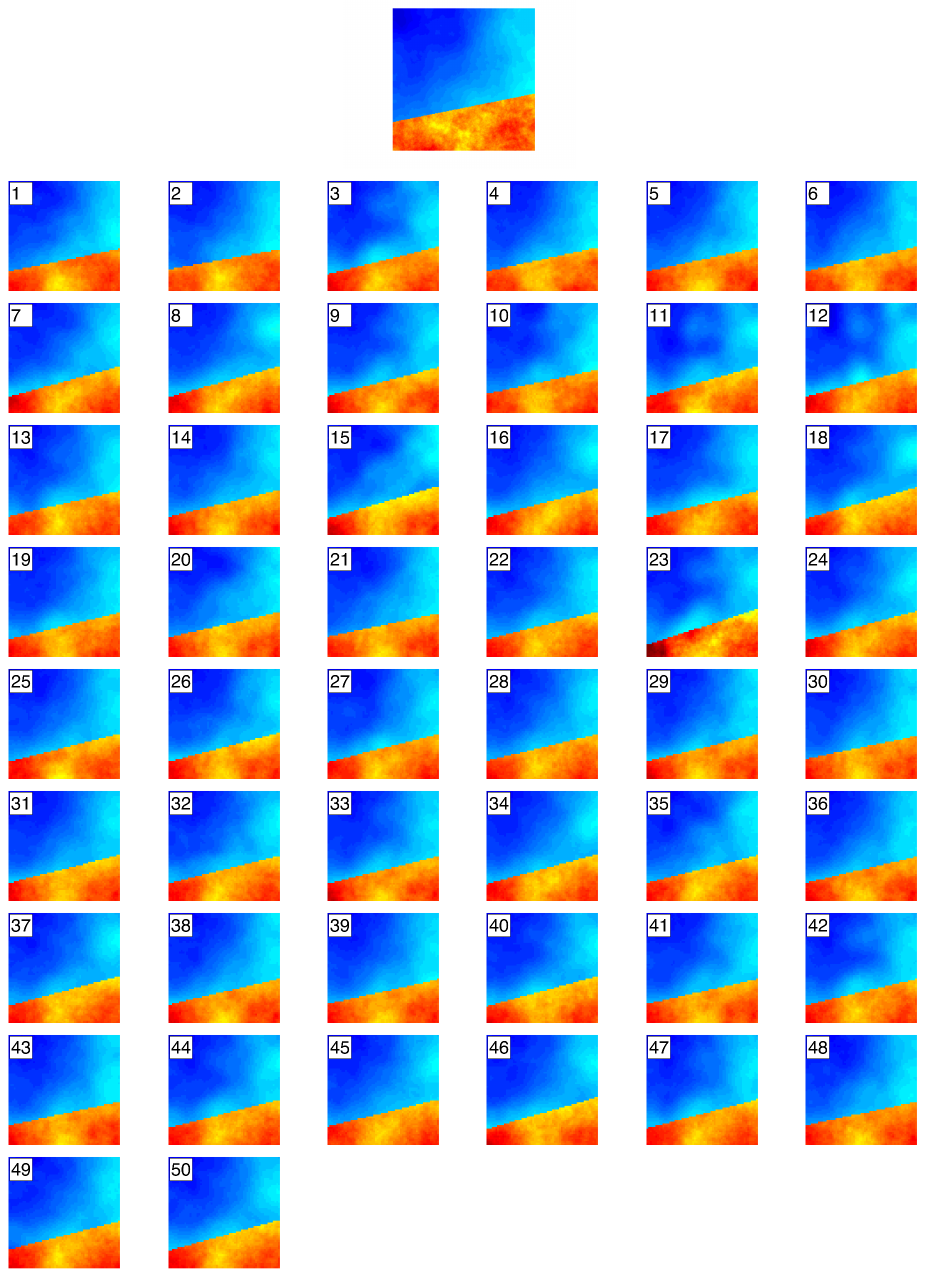}\\
\end{center}
\vspace{-0.5cm}
\caption{(Model 1) the true log-permeability field (top), and the conditional mean arising from MCMC chains initialised at the corresponding local minimisers above.}
\label{cm1}
\end{figure}

\begin{figure}[h!]
\begin{center}
\includegraphics[width=\textwidth,trim=2.6cm 3cm 3cm 3cm]{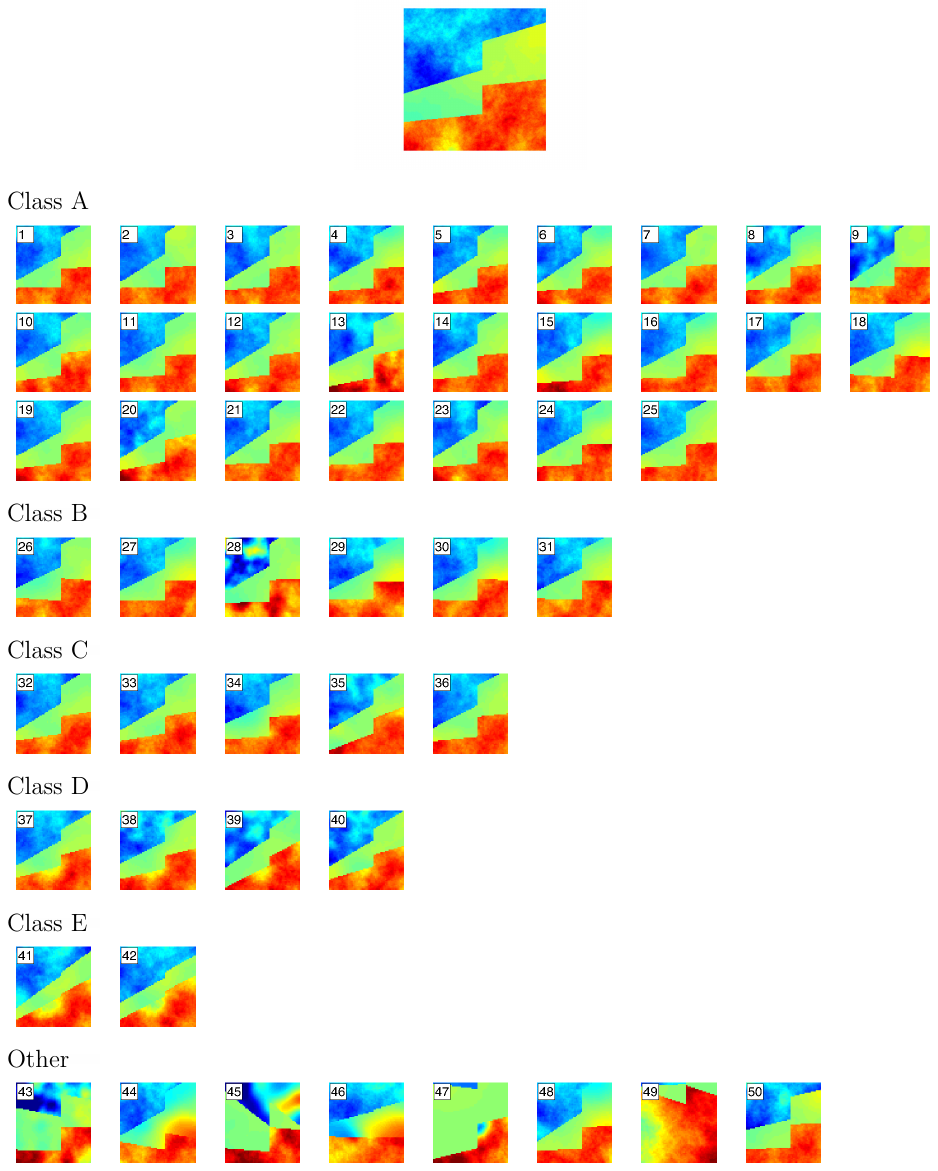}\\
\end{center}
\vspace{-2cm}
\caption{(Model 2) the true log-permeability field (top), and the conditional mean arising from MCMC chains initialised at the corresponding local minimisers above. We group them into the same classes as the local minimisers.}
\label{cm2}
\end{figure}

\begin{figure}[h!]
\begin{center}
\includegraphics[width=\textwidth,trim=2.6cm 3cm 3cm 3cm]{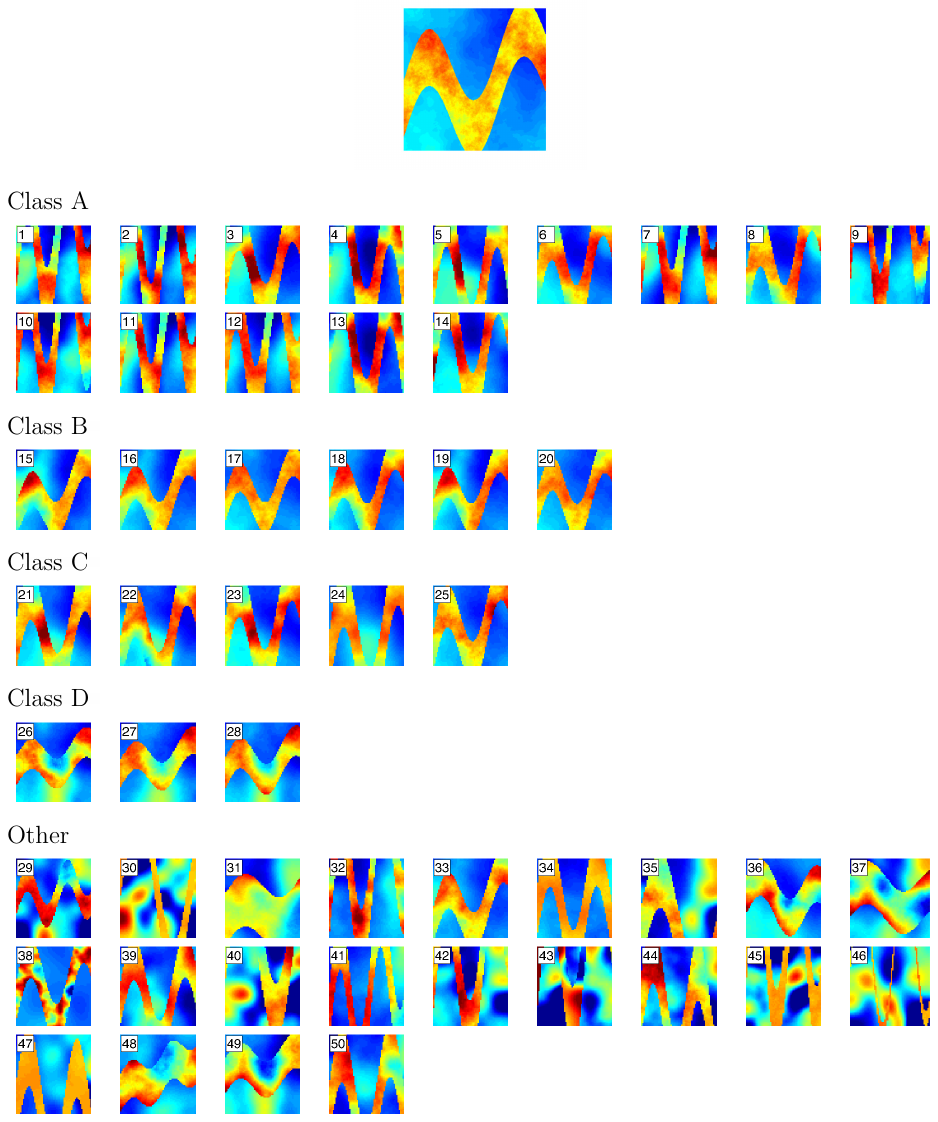}\\
\end{center}
\vspace{-2.7cm}
\caption{(Model 3) the true log-permeability field (top), and the conditional mean arising from MCMC chains initialised at the corresponding local minimisers above. We group them into the same classes as the local minimisers.}
\label{cm3}
\end{figure}

%%%%%%
% PHI TRACES 
\begin{figure}[h!]
\begin{center}
\includegraphics[width=\textwidth,trim=2.6cm 3cm 3cm 3cm]{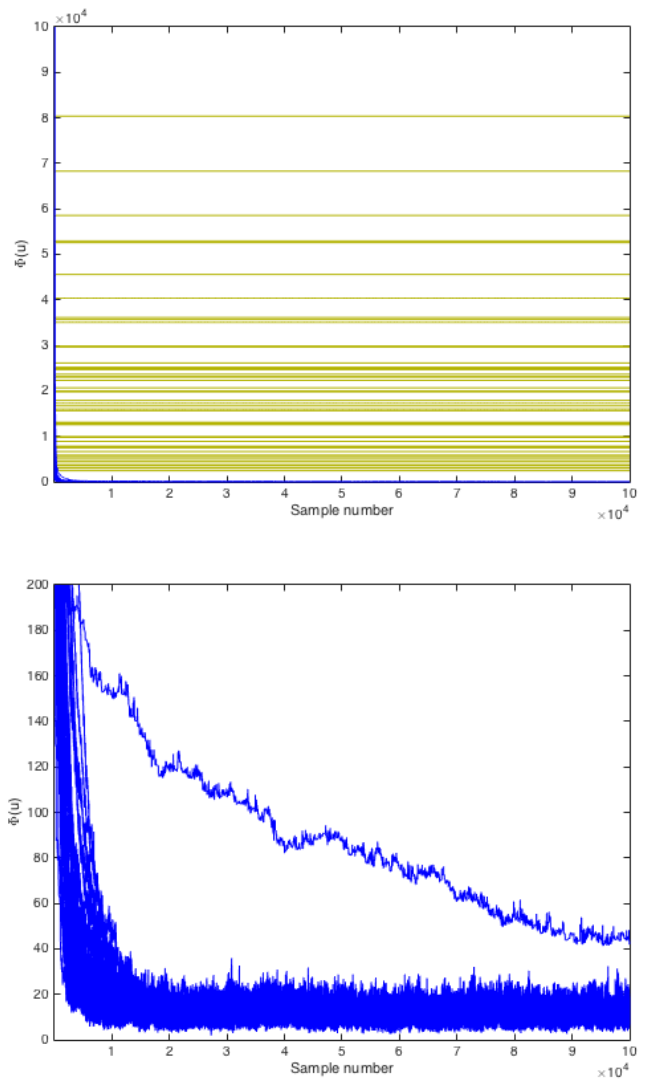}\\
\end{center}
\vspace{-3cm}
\caption{(Model 1) The evolution of $\Phi$ as the MCMC chains progress. The horizontal lines represent the value of each local minimiser under $\Phi$. Nearly all of the simulations find a small value of $\Phi$ almost immediately, but simulation \new{23} remains caught in the local minimiser for some time before it follows.}
\label{phitrace1}
\end{figure}

\begin{figure}[h!]
\begin{center}
\includegraphics[width=\textwidth,trim=2.6cm 3cm 3cm 3cm]{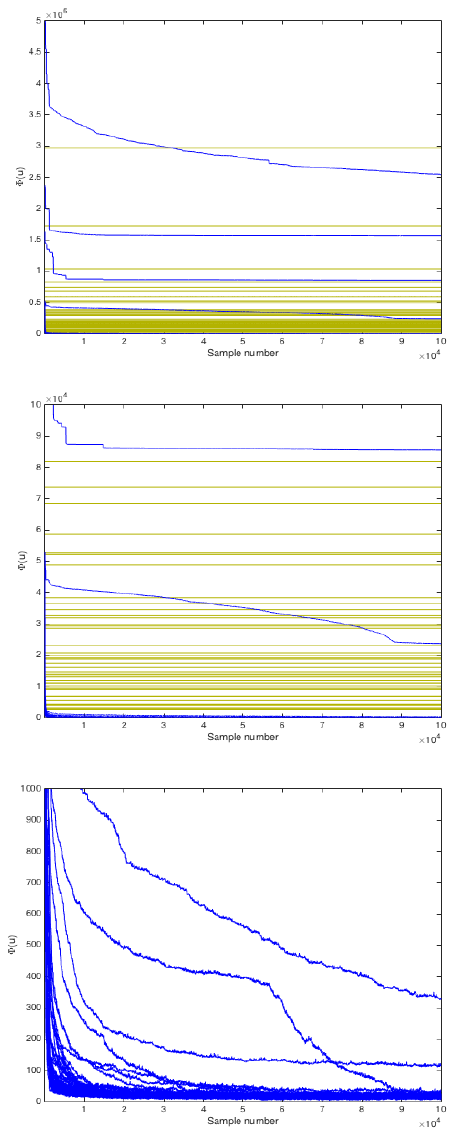}\\
\end{center}
\vspace{-2.5cm}
\caption{(Model 2) The evolution of $\Phi$ as the MCMC chains progress. The horizontal lines represent the value of each local minimiser under $\Phi$. The majority of the simultions find a small value of $\Phi$ almost immediately, but numerous fail to reach there, settling in local minima. The shape of these minima can be seen in Figure \ref{cm2}, and generally correspond to those in the same class as the initial state.}
\label{phitrace2}
\end{figure}

\begin{figure}[h!]
\begin{center}
\includegraphics[width=\textwidth,trim=2.6cm 3cm 3cm 3cm]{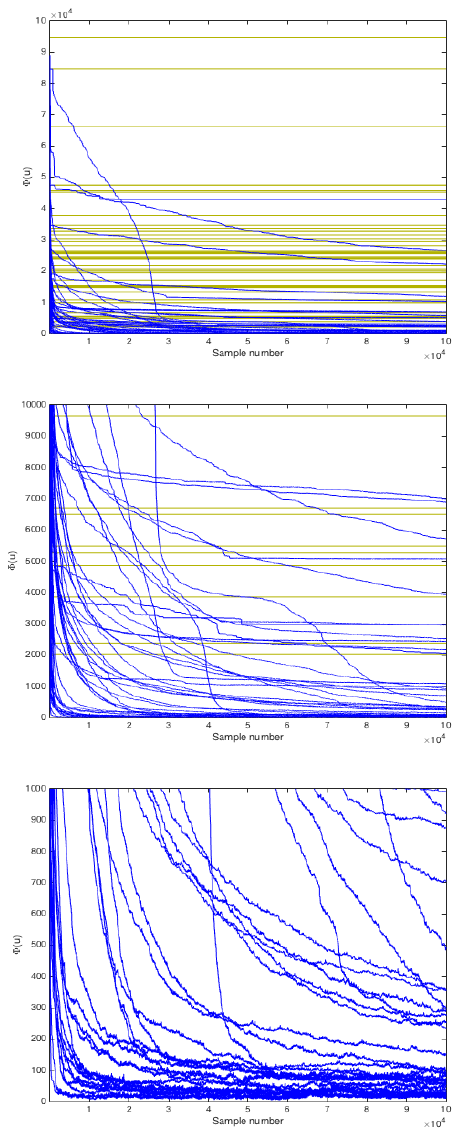}\\
\end{center}
\vspace{-2.5cm}
\caption{(Model 3) The evolution of $\Phi$ as the MCMC chains progress. The horizontal lines represent the value of each local minimiser under $\Phi$. The majority of the simultions find a small value of $\Phi$ almost immediately, but numerous fail to reach there, settling in local minima. The shape of these minima can be seen in Figure \ref{cm3}, and generally correspond to those in the same class as the initial state.}
\label{phitrace3}
\end{figure}

%%%%%%
% NEAREST MINIMISERS

\begin{figure}[h!]
\begin{center}
\includegraphics[width=\textwidth,trim=2.6cm 3cm 3cm 3cm]{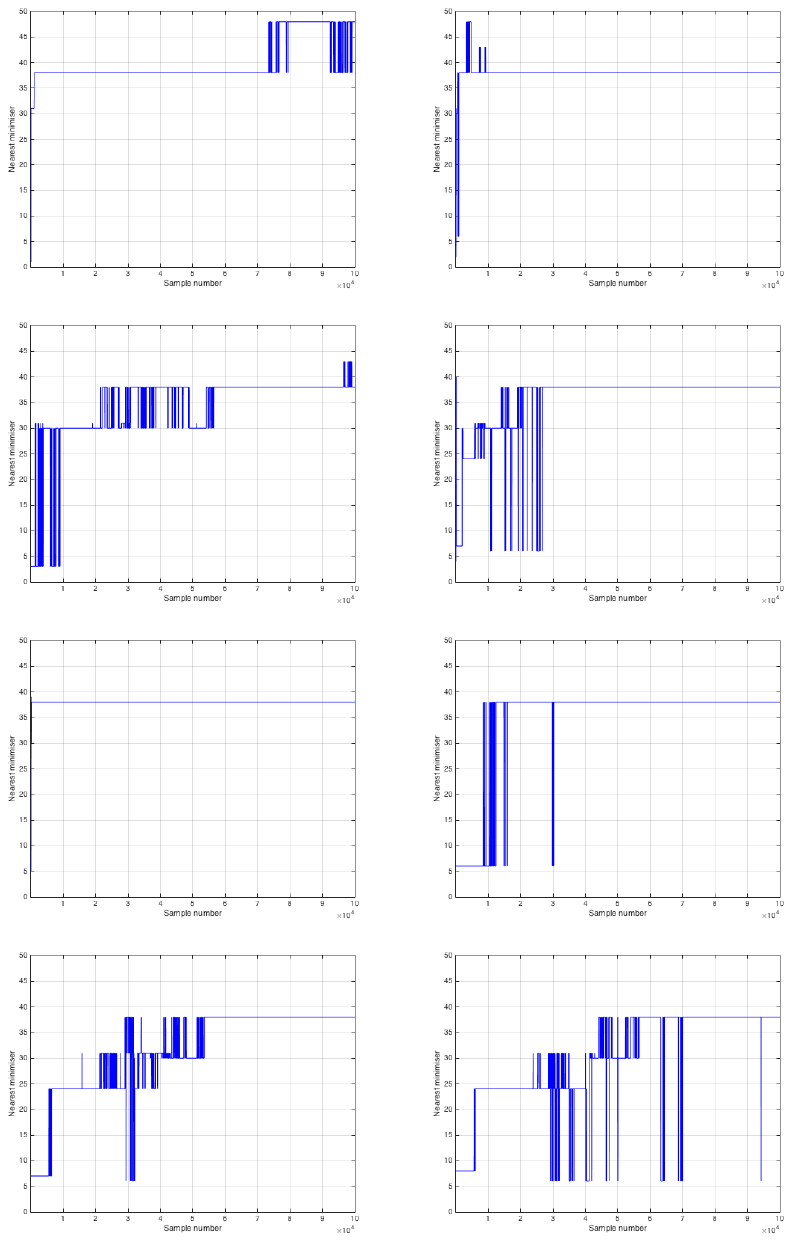}
\end{center}
\caption{(Model 1) The trace of $m_n$ as defined by (\ref{eq:mn}), when the chain is initialised at a variety of minimisers -- specifically numbers $1,2,\ldots,8$.}
\label{argmin1}
\end{figure}

\begin{figure}[h!]
\begin{center}
\includegraphics[width=\textwidth,trim=2.6cm 3cm 3cm 3cm]{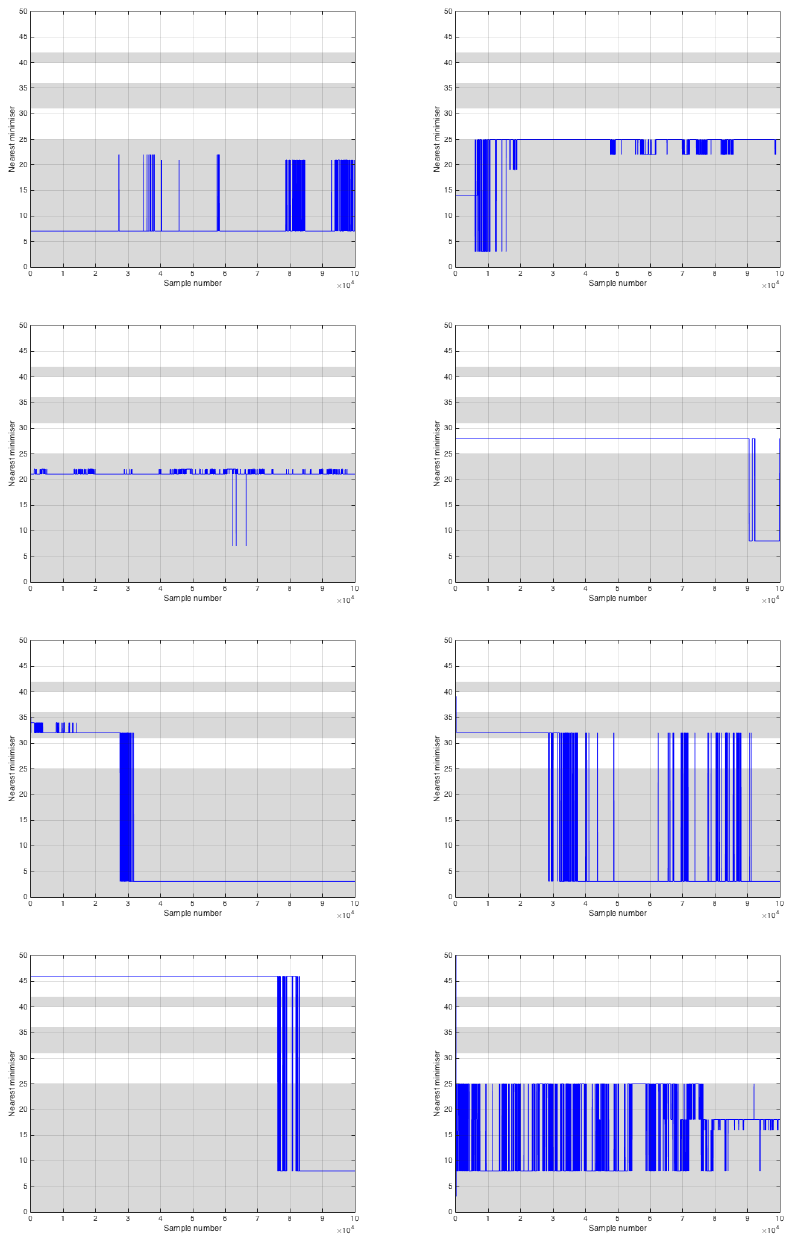}
\end{center}
\caption{(Model 2) The trace of $m_n$ as defined by (\ref{eq:mn}), when the chain is initialised at a variety of minimisers -- specifically numbers $7,14,21,28,35,39,46$ and $50$. The different classes are alternately shaded.}
\label{argmin2}
\end{figure}

\begin{figure}[h!]
\begin{center}
\includegraphics[width=\textwidth,trim=2.6cm 3cm 3cm 3cm]{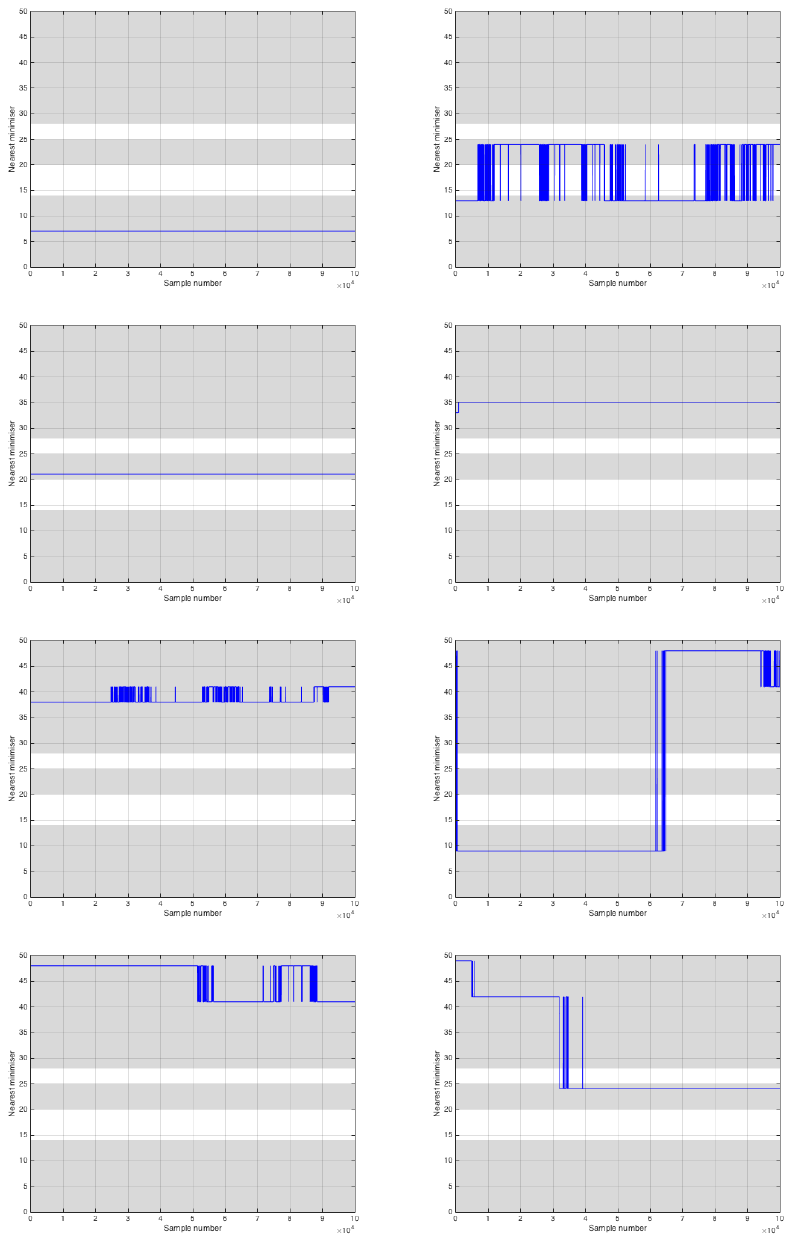}
\end{center}
\caption{(Model 3) The trace of $m_n$ as defined by (\ref{eq:mn}), when the chain is initialised at a variety of minimisers -- specifically numbers $7,13,21,33,38,47,48$ and $49$. The different classes are alternately shaded.}
\label{argmin3}
\end{figure}

%\appendix

\section{Conclusions and Future Work}
\label{sec:concl}

\mmd{
We have made a new contribution to the recently developed theory
of MAP estimation in infinite dimensions \cite{map,wmap}. We link
MAP estimation to a variational Onsager-Machlup functional. The
work is focused on  priors for piecewise Gaussian random fields,
with random interfaces parameterised finite-dimensionally. Such
fields arise naturally in applications such as groundwater flow
and EIT, and these are used to illustrate the theory and numerics.
The work opens up several new avenues for investigation. A major
theoretical direction is to fully reconcile the approaches
in \cite{map} and \cite{wmap}; the work in this paper suggests that
this may be possible. On the applications side an important
new direction would be to consider problems in which the  
geometric parameters are no longer independent from the fields a priori. A possible extension could be to treat the geometric parameters as hyperparameters for the fields under the prior. This would allow, for example, the fields to have specific boundary conditions at the interfaces, which may be more physically appropriate in some contexts. A related hierarchical model was considered in \cite{NF98}, in which prior samples were piecewise white; this could be extended to allow for spatial correlations in the continuum setting. Computationally an exciting direction is to build upon definitions of
MAP estimators to develop hybrid algorithms which fully exploit
local minimiser structure of the Onsager-Machlup functional within
MCMC.
}

\section{Appendix}
In this appendix we provide proofs of the results given in the paper.

\subsection{Results From Section \ref{sec:fwd}}
Before we prove Lemma \ref{lem:cont_darcy} we require the following lemma.
\begin{lemma}
\label{lem:dct2}
Let $(X,\mathcal{F},\mu)$ be a measure space and $f \in L^1(X,\mathcal{F},\mu)$. Let $B_n\subseteq \mathcal{F}$ be a sequence of measurable subsets of $X$ with $\mu(B_n)\rightarrow 0$ as $n\rightarrow\infty$. Then
\begin{eqnarray*}
\int_{B_n}f(x)\,\mu(\dee x)\rightarrow 0\;\;\;\text{as }n\rightarrow \infty.
\end{eqnarray*}
\end{lemma}
\begin{proof}
Write $f_n(x) = f(x)\mathds{1}_{B_n}(x)$. We have that $f_n\rightarrow 0$ in measure: for any $\delta > 0$,
\begin{eqnarray*}
\mu(\{x \in X\;|\;|f_n(x)| > \delta\}) \leq \mu(\{x \in X\;|\;|f_n(x)|\neq 0\}) \leq \mu(B_n)\rightarrow 0.
\end{eqnarray*}
Now suppose that $\|f_n\|_{L^1}$ does not tend to zero. Then there exists $\delta > 0$ and a subsequence $(f_{n_k})_{k\geq 1}$ such that $\|f_{n_k}\|_{L^1} \geq \delta$ for all $k \geq 1$. This subsequence still converges to zero in measure, and so admits a further subsequence that converges to zero almost surely. We can bound this subsequence above uniformly by $f$, and so an application of the dominated convergence theorem leads to a contradiction. The result follows.
\end{proof}

\begin{proof}[Proof of Lemma \ref{lem:cont_darcy}]
Showing that $\mathcal{R}$ is well-defined is equivalent to showing that PDE (\ref{pde2}) has a unique solution for all $(u,a) \in X\times\Lambda$.  Since $u^a \in L^\infty(D)$ it is bounded, and so by the continuity and positivity of $\sigma$ there exist $\kappa_{\min},\kappa_{\max} > 0$ with $\kappa_{\min} \leq \sigma(u^a) \leq \kappa_{\max}$. The associated bilinear form is hence bounded and coercive. The right hand side can be seen to lie in $H^{-1}(D)$ since $G \in H^1(D)$ and $\sigma(u^a) \leq \kappa_{\max}$, and so a unique solution exists by Lax-Milgram.
\begin{enumerate}[(i)]
\item In its weak form, the PDE (\ref{pde2}) is given by
\begin{eqnarray*}
\int_D \sigma(u^a)\nabla q_{u,a}\cdot\nabla\varphi = f(\varphi) - \int_D \sigma\nabla G\cdot\nabla\varphi\;\;\;\text{for all }\varphi \in V.
\end{eqnarray*}
Taking $\varphi = q_{u,a}$ we deduce that
\begin{eqnarray*}
\kappa_{\min}(u,a)\|\nabla q_{u,a}\|_{L^2}^2 &\leq \int_D\sigma(u^a)\nabla q_{u,a}\cdot\nabla q_{u,a}\\
&= f(q_{u,a}) - \int_D \sigma(u^a)\nabla G\cdot\nabla q_{u,a}\\
&\leq \|f\|_{V^*}\|q_{u,a}\|_V + \|\sigma(u^a)\|_{L^\infty}\|\nabla G\|_{L^2}\|\nabla q_{u,a}\|_{L^2}
\end{eqnarray*}
and so we have the estimate
\begin{eqnarray*}
\|p_{u,a}\|_V &\leq \|q_{u,a}\|_V + \|G\|_V\\
&\leq (\|f\|_{V^*} + \|\sigma(u^a)\|_{L^\infty}\|G\|_V)/\kappa_{\min}(u,a) + \|G\|_V.
\end{eqnarray*}

\item Let $u,v \in X$ and $a \in \Lambda$. Then $p_{u,a}-p_{v,a}$ satisfies the PDE
\begin{eqnarray*}
\begin{cases}
-\nabla\cdot\big(\sigma(u^a)\nabla(p_{u,a}-p_{v,a})\big) &= \nabla \cdot\big((\sigma(u^a) - \sigma(v^a))\nabla p_{v,a}\big)\;\;\text{ in }D\\
\hspace{2.7cm}p_{u,a} - p_{v,a} &= 0\hspace{3.65cm}\text{ on }\partial D.
\end{cases}
\end{eqnarray*}
Setting $\kappa_*(u,v,a) = \kappa_{\min}(u,a)\wedge\kappa_{\min}(v,a)$, we see
\begin{eqnarray*}
\fl\kappa_*(u,v,a)\|\nabla(p_{u,a} - p_{v,a})\|_{L^2}^2 &\leq \int_D \sigma(u^a)|\nabla(p_{u,a} - p_{v,a})|^2\\
&= \int_D(\sigma(u^a) - \sigma(v^a))\nabla(p_{u,a}-p_{v,a})\cdot\nabla p_{v,a}\\
&\leq \|\sigma(u^a) - \sigma(v^a)\|_{L^\infty}\|\nabla(p_{u,a}-p_{v,a})\|_{L^2}\|\nabla p_{v,a}\|_{L^2}
\end{eqnarray*}
and so by (i),
\begin{eqnarray*}
\fl\|p_{u,a} - p_{v,a}\|_V &\leq \|p_{v,a}\|_V\|\sigma(u^a)-\sigma(v^a)\|_{L^\infty}/\kappa_*(u,v,a)\\
&\leq \|\sigma(u^a) - \sigma(v^a)\|_{L^\infty}\\
&\hspace{1cm}\times \big((\|f\|_{V^*} + \|\sigma(u^a)\|_{L^\infty}\|G\|_V)/\kappa_{*}(u,a)^2 + \|G\|_V/\kappa_{*}(u,a)\big).
\end{eqnarray*}
Using that the $A_i$ are disjoint gives that
\begin{eqnarray*}
\|\sigma(u^a) - \sigma(v^a)\|_{L^\infty} &= \left\|\sigma\left(\sum_{i=1}^N u_i\mathds{1}_{A_i(a)}\right) - \sigma\left(\sum_{i=1}^N v_i\mathds{1}_{A_i(a)}\right)\right\|_{L^\infty}\\
&=\|\sigma(u_k) - \sigma(v_k)\|_{L^\infty}
\end{eqnarray*}
for some $k = k(a)$. Now suppose that $\|u\|_X,\|v\|_X < r$. Then the $C^1$ property of $\sigma$ yields
\begin{eqnarray*}
\fl\|\sigma(u_k)-\sigma(v_k)\|_{L^\infty} \leq \max_{|t|\leq r}|\sigma'(t)|\cdot\|u_k-v_k\|_{L^\infty} \leq \max_{|t|\leq r}|\sigma'(t)|\cdot\|u-v\|_X.
\end{eqnarray*}
Finally we deal with the $\kappa_*^{-j}$ terms:
\begin{eqnarray*}
\kappa_*(u,v,a)^{-j} &= \left[\left(\underset{x\in D}{\mathrm{essinf}}\,e^{u^a(x)}\right)\wedge\left(\underset{x\in D}{\mathrm{essinf}}\,e^{v^a(x)}\right)\right]^{-j}\\
& \leq \left(\min_{|t|\leq r}\sigma(t)\wedge \min_{|t|\leq r}\sigma(t)\right)^{-j}\\
&= \left(\min_{|t|\leq r}\sigma(t)\right)^{-j}.
\end{eqnarray*}
We bound the $\|\sigma(u^a)\|_{L^\infty}$ term similarly. Putting the above bounds together, we have
\begin{eqnarray*}
\fl\|\mathcal{R}(u,a) - \mathcal{R}(v,a)\|_V &= \|p_{u,a}-p_{v,a}\|_V\\
&\leq \max_{j=1,2}\left(\min_{|t|\leq r}\sigma(t)\right)^{-j}\left(\|f\|_{V^*} + \|G\|_V\left(\max_{|t|\leq r}\sigma(t)+1\right)\right)\\
&\hspace{1cm}\times\max_{|t|\leq r}|\sigma'(t)|\cdot\|u-v\|_X\\
&= L(r)\|u-v\|_X.
\end{eqnarray*}
Note that the constant $L(r)$ is uniform in $a$.
\item We use a similar approach to the previous part. Given $u \in X$ and $a,b \in \Lambda$, the difference $p_{u,a}-p_{u,b}$ satisfies
\begin{eqnarray*}
\begin{cases}
-\nabla\cdot(\sigma(u^a)\nabla(p_{u,a}-p_{u,b})) &= \nabla \cdot((\sigma(u^a) - \sigma(u^b))\nabla p_{u,b})\;\;\text{ in }D\\
\hspace{2.7cm}p_{u,a} - p_{u,b} &= 0\hspace{3.65cm}\text{ on }\partial D
\end{cases}
\end{eqnarray*}
which leads to the bound
\begin{eqnarray*}
\fl\kappa_\dagger(u,a,b)\|\nabla(p_{u,a} - p_{u,b})\|_{L^2}^2 &\leq \int_D \sigma(u^a)|\nabla(p_{u,a} - p_{u,b})|^2\\
&= \int_D(\sigma(u^a) - \sigma(u^b))\nabla(p_{u,a}-p_{u,b})\cdot\nabla p_{u,b}\\
&\leq \|\nabla(p_{u,a}-p_{u,b})\|_{L^2}\|(\sigma(u^a)-\sigma(u^b))\nabla p_{u,b}\|_{L^2}
\end{eqnarray*}
where $\kappa_\dagger(u,a,b) = \kappa_{\min}(u,a)\wedge\kappa_{\min}(u,b)$. It follows that
\begin{eqnarray*}
\|p_{u,a} - p_{u,b}\|_V &\leq \|(\sigma(u^a)-\sigma(u^b))\nabla p_{u,b}\|_{L^2}/\kappa_\dagger(u,a,b).
\end{eqnarray*}

Again by the disjointness of the $A_i$ and the $C^1$ property of $\sigma$,
\begin{eqnarray*}
\fl\|(\sigma(u^a) - \sigma(u^b))\nabla p_{u,b}\|_{L^2} &= \left\|\left(\sigma\left(\sum_{i=1}^N u_i\mathds{1}_{A_i(a)}\right) - \sigma\left(\sum_{i=1}^N u_i\mathds{1}_{A_i(b)}\right)\right)\nabla p_{u,b}\right\|_{L^2}\\
&= \left\|\sum_{i=1}^N\left(\sigma\left(u_i\mathds{1}_{A_i(a)}\right) - \sigma\left(u_i\mathds{1}_{A_i(b)}\right)\right)\nabla p_{u,b}\right\|_{L^2}\\
&\leq \sum_{i=1}^N\left\|\left(\sigma\left(u_i\mathds{1}_{A_i(a)}\right) - \sigma\left(u_i\mathds{1}_{A_i(b)}\right)\right)\nabla p_{u,b}\right\|_{L^2}\\
&\leq \sum_{i=1}^N \max_{|t|\leq\|u_i\|_\infty}|\sigma'(t)|\cdot\left\|\left|u_i\mathds{1}_{A_i(a)} - u_i\mathds{1}_{A_i(b)}\right|\nabla p_{u,b}\right\|_{L^2}\\
&\leq \sum_{i=1}^N \max_{|t|\leq\|u_i\|_\infty}|\sigma'(t)|\cdot\|u_i\|_\infty\left\|\mathds{1}_{A_i(a)\Delta A_i(b)}\nabla p_{u,b}\right\|_{L^2}
\end{eqnarray*}
since $|\mathds{1}_A - \mathds{1}_B| = \mathds{1}_{A\Delta B}$. Now as before we can bound $\kappa_\dagger^{-1}$:
\begin{eqnarray*}
\kappa_\dagger(u,v,a)^{-1} &= \left[\left(\underset{x\in D}{\mathrm{essinf}}\,e^{u^a(x)}\right)\wedge\left(\underset{x\in D}{\mathrm{essinf}}\,e^{u^b(x)}\right)\right]^{-1}\\
&\leq \left(\min_{|t|\leq \max \|u_i\|_\infty}\sigma(t)\wedge \min_{|t|\leq \max \|u_i\|_\infty}\sigma(t)\right)^{-1}\\
&\leq \left(\min_{|t|\leq \|u\|_X}\sigma(t)\right)^{-1}.
\end{eqnarray*}
Putting the above bounds together, we have
\begin{eqnarray*}
\fl\|\mathcal{R}(u,a)&-\mathcal{R}(u,b)\|_V = \|p_{u,a}-p_{u,b}\|_V\\
\fl&\leq \left(\min_{|t|\leq \|u\|_X}\sigma(t)\right)^{-1}\sum_{i=1}^N \max_{|t|\leq\|u_i\|_\infty}|\sigma'(t)|\cdot\|u_i\|_\infty\left\|\mathds{1}_{A_i(a)\Delta A_i(b)}\nabla p_{u,b}\right\|_{L^2}\\
\fl&\leq \left(\min_{|t|\leq \|u\|_X}\sigma(t)\right)^{-1}\sum_{i=1}^N \max_{|t|\leq\|u_i\|_\infty}|\sigma'(t)|\cdot\|u_i\|_\infty\left(\int_{A_i(a)\Delta A_i(b)}|\nabla p_{u,b}|^2\right)^{1/2}.
\end{eqnarray*}

The right hand goes to zero as each $|A_i(a)\Delta A_i(b)| \rightarrow 0$ by Lemma \ref{lem:dct2}, since $|\nabla p_{u,b}| \in L^2(D)$, and so the continuity of $\mathcal{R}(u,\cdot)$ follows from the assumed continuity of the maps $A_i$.
\end{enumerate}
\end{proof}

\begin{proof}[Proof of Proposition \ref{prop:cont_eit}]
\begin{enumerate}
\item Theorem 2.3 in \cite{frechet} tells us that the mapping from the conductivity to the weak solution of (\ref{eq:eitpde}) is Fr\'echet differentiable with respect to the supremum norm, and hence locally Lipschitz. Note that the mapping from the solution to the boundary voltage measurements, $(v,V) \mapsto V$, is smooth, and the assumptions on $\sigma$ imply that it is Lipschitz. It hence suffices to show that the mapping $u\mapsto F(u,a)$ is Lipschitz for each $a \in \Lambda$. Let $u,v \in X$ and $a \in \Lambda$, then
\begin{eqnarray*}
\|F(u,a) - F(v,a)\|_\infty \leq \sum_{i=1}^N \|u_i - v_i\|_\infty\mathds{1}_{A_i(a)} \leq C\|u - v\|_X
\end{eqnarray*}
and the result follows.
\item By Corollary 2.8 in \cite{eit} and the continuity of $\sigma$, it suffices to show that $a_n\rightarrow a$ in $\Lambda$ implies that $F(u,a_n)\rightarrow F(u,a)$ in measure. For any $p \in (1,\infty)$ we have that
\begin{eqnarray*}
\int_D |F(u,a_n) - F(u,a)|^p\,\dee x &\leq \sum_{i=1}^N \int_D |u_i|^p\mathds{1}_{A_i(a_n)\Delta A_i(a)}\,\dee x\\
&\leq \sum_{i=1}^N\|u_i\|_\infty^p\cdot |A_i(a_n)\Delta A_i(a)|
\end{eqnarray*}
From the assumed continuity of $A_i(\cdot)$ it follows that $F(u,a_n)\rightarrow F(u,a)$ in $L^p$ for any $p \in (1,\infty)$, and hence in measure.
\end{enumerate}
\end{proof}

\subsection{Results From Section \ref{sec:post}}

\begin{proof}[Proof of Theorem \ref{thm:post_exist}]
\begin{enumerate}[(i)]
\item We first claim that the assumptions on $\Phi$ mean that $\Phi(\cdot,\cdot;y):X'\times\Lambda'\rightarrow\mathbb{R}$ is continuous for each $y \in Y$. Fix $y \in Y$ and $(u,a) \in X'\times\Lambda'$. Choose any approximating sequence $(u_n,a_n)_{n\geq 1}\subseteq X'\times\Lambda'$ such that $(u_n,a_n)\rightarrow(u,a)$. Then the assumptions on the norm on $X\times\Lambda$ means that $\|u_n-u\|_X\rightarrow 0$ and $|a_n-a|\rightarrow 0$. Letting $r > \max\{\|u\|_X,\sup_n\|u_n\|_X\}$, we may approximate
\begin{eqnarray*}
\fl|\Phi(u_n,a_n;y)-\Phi(u,&a;y)| \leq |\Phi(u_n,a_n;y)-\Phi(u,a_n;y)| + |\Phi(u,a_n;y)-\Phi(u,a;y)|\\
\fl&\leq M_3(r,a_n,y)\|u_n-u\|_X + |\Phi(u,a_n;y)-\Phi(u,a;y)|\\
\fl&\leq \left(\sup_{k\in\mathbb{N}}M_3(r,a_k,y)\right)\cdot\|u_n-u\|_X + |\Phi(u,a_n;y)-\Phi(u,a;y)|
\end{eqnarray*}
where the supremum is finite due the continuity of $M_3$ in its second component. Since $\Phi$ is also continuous in its second component, we see that the right-hand side tends to zero as $(u_n,a_n)\rightarrow (u,a)$.

Now as $\Phi(\cdot,\cdot;y):X'\times\Lambda'\rightarrow\mathbb{R}$ is continuous and $(\mu_0\times\nu_0)(X'\times\Lambda') = 1$, $\Phi(\cdot,\cdot;y)$ is $\mu_0\times\nu_0$-measurable.
\begin{comment}
Let $y \in Y$. Since $X'$ is metrisable and $\mathbb{R}$ is a locally convex topological vector space, a theorem of Rudin \cite{rudin} tells us that there exists a sequence of continuous functions $\Phi_n(\cdot,\cdot;y):X'\times\Lambda'\rightarrow Y$ converging to $\Phi(\cdot,\cdot;y)$ pointwise. Since $(\mu_0\times\nu_0)(X'\times\Lambda') = 1$, it follows that each $\Phi_n(\cdot,\cdot;y)$ is $\mu_0\times\nu_0$-measurable. Then so are all of its components $\Phi_n^{(j)}(\cdot,\cdot;y):X'\times\Lambda'\rightarrow\mathbb{R}$, $j=1,\ldots,J$. Now given any $\alpha \in \mathbb{R}$, we have for each $j$,
\begin{eqnarray*}
\Phi^{(j)}(\cdot,\cdot;y)^{-1}((-\infty,\alpha]) &= \left\{ (u,a)\;|\;\lim_{n\rightarrow\infty}\Phi_n^{(j)}(u,a;y) \leq \alpha\right\}\\
&= \left\{ (u,a)\;|\;\limsup_{n\rightarrow\infty}\Phi_n^{(j)}(u,a;y) \leq \alpha\right\}\\
&= \bigcap_{n=1}^\infty\bigcup_{m=n}^\infty \Phi_n^{(j)}(\cdot,\cdot;y)^{-1}((-\infty,\alpha])
\end{eqnarray*}
and so each $\Phi^{(j)}(\cdot,\cdot;y)$ is $\mu_0\times\nu_0$-measurable. Therefore $\Phi(\cdot,\cdot;y)$ is $\mu_0\times\nu_0$-measurable.
\end{comment}
Setting $Z = X'\times\Lambda'$, we can consider $\Phi:Z\times Y\rightarrow\mathbb{R}$. This is a Caratheodory function, and it is known that these are jointly measurable, see for example \cite{hitchhiker}. We conclude that $\Phi$ is $\mu_0\times\nu_0\times\mathbb{Q}_0$ measurable.

\item We first show $Z(y)$ is finite. Since $\mu_0$ is Gaussian, by Fernique's theorem there exists $\alpha > 0$ such that
\begin{eqnarray*}
\int_{X}\exp(\alpha\|u\|_X^2)\,\mu_0(\dee u) < \infty.
\end{eqnarray*}
Then using Assumptions \ref{assump:map}(i), we have the lower bound
\begin{eqnarray*}
\Phi(u,a;y) \geq M_1(\alpha) - \alpha\|u\|_X^2
\end{eqnarray*}
from which we conclude that $Z(y) < \infty$.

Now fix $r > 0$. Let $y \in Y$ and take $(u,a) \in X'\times\Lambda'$ with $\max\{\|u\|_X,|a|\} < r$. Then we have by the local Lipschitz property
\begin{eqnarray*}
\fl|\Phi(u,a;y)| \leq M_3(r,y)\|u\|_X + |\Phi(0,a;y)| \leq M_3(r,a,y)r + |\Phi(0,a;y)|.
\end{eqnarray*}
Using the continuity of $\Phi$ and $M_3$ in $a$, we can maximise the right hand side over $|a| < r$ to deduce that
\begin{eqnarray*}
|\Phi(u,a;y)| \leq K(r,y) .
\end{eqnarray*}
Thus $\Phi(\cdot,\cdot;y)$ is bounded on bounded sets.

Now we can proceed as in \cite{lecturenotes}. Using that $(\mu_0\times\nu_0)(X'\times\Lambda') = 1$, we have that
\begin{eqnarray*}
Z(y) = \int_{X'\times\Lambda'}\exp(-\Phi(u,a;y))\,\mu_0(\dee u)\nu_0(\dee a).
\end{eqnarray*}
Set $B' = (X'\times\Lambda')\cap B$, and set
\begin{eqnarray*}
R = \sup\{\max\{\|u\|_X,|a|\}\,|\,(u,a) \in B'\}.
\end{eqnarray*}
We deduce that
\begin{eqnarray*}
\sup_{(u,a) \in B'}\Phi(u,a;y) \leq K(R,y) < \infty
\end{eqnarray*}
and so
\begin{eqnarray*}
\fl Z(y) \geq \int_{B'}\exp(-K(R,y))\,\mu_0(\dee u)\nu_0(\dee a) = \exp(-K(R,y))(\mu_0\times\nu_0)(B') > 0.
\end{eqnarray*}
Hence the measure $\mu^y$ is well-defined.
\item The well-posedness of the posterior is proved in virtually the same way as Theorem 4.5 in \cite{lecturenotes}.
\end{enumerate}
\end{proof}

%%%%%%%%%%%%%%%%%
%%% SECTION 5 %%%
%%%%%%%%%%%%%%%%%

\subsection{Results From Section \ref{sec:map}}
Throughout this section, for $\delta > 0$ and $(u,a) \in X\times\Lambda$, we will denote $\mathcal{J}^\delta(u,a) = \mu(B^\delta(u,a))$. To prove Theorems \ref{omthm} and \ref{exist}, we first require two lemmas.

\begin{lemma}
\label{ombasic}
Let $(u_1,a_1), (u_2,a_2) \in E\times \intr(S)$. Then
\begin{eqnarray*}
\lim_{\delta\downarrow 0} \frac{(\mu_0\times\nu_0)(B^\delta(u_1,a_1))}{(\mu_0\times\nu_0)(B^\delta(u_2,a_2))} &= e^{\frac{1}{2}\|u_2\|_E^2 - \frac{1}{2}\|u_1\|_E^2}\cdot\frac{\rho(a_1)}{\rho(a_2)}\\
&= \exp\left(J(u_2) + K(a_2) - J(u_1) - K(a_1)\right).
\end{eqnarray*}
\end{lemma}

\begin{proof}
We adapt the proof of Proposition 18.3 in \cite{lifshits} to first show that
\begin{eqnarray*}
(\mu_0\times\nu_0)(B^\delta(u_1,a_1)) \sim e^{-\frac{1}{2}\|u_1\|_E^2}(\mu_0\times\nu_0)(B^\delta(0,a_1))\;\;\;\text{ as }\delta\downarrow 0.
\end{eqnarray*}
The first half of the proof is almost identical to that in \cite{lifshits}, though some care must be taken since we cannot (a priori) separate the integrals over balls in $X\times\Lambda$ into products of those over balls in $X$ and $\Lambda$. Using the Cameron-Martin theorem we see that
\begin{eqnarray*}
(\mu_0\times\nu_0)(B^\delta(u_1,a_1)) = e^{-\frac{1}{2}\|u_1\|_E^2}\int_{B^\delta(0,a_1)}e^{\langle u_1,u\rangle_E}\,\mu_0(\dee u)\nu_0(\dee a).
\end{eqnarray*}
Since $\langle u_1,-u\rangle_E = -\langle u_1,u\rangle_E$ and $B^\delta(0,a_1)$ is symmetric about $0 \in X$, it follows that
\begin{eqnarray*}
\int_{B^\delta(0,a_1)}e^{\langle u_1,u\rangle_E}\,\mu_0(\dee u)\nu_0(\dee a) &= \int_{B^\delta(0,a_1)}\frac{1}{2}\left(e^{\langle u_1,u\rangle_E}+e^{-\langle u_1,u\rangle_E}\right)\,\mu_0(\dee u)\nu_0(\dee a)\\
&\geq (\mu_0\times\nu_0)(B^\delta(0,a_1))
\end{eqnarray*}
which gives the inequality
\begin{eqnarray}
\label{thm1lower}
(\mu_0\times\nu_0)(B^\delta(u_1,a_1)) \geq e^{-\frac{1}{2}\|u_1\|_E^2}(\mu_0\times\nu_0)(B^\delta(0,a_1)).
\end{eqnarray}
For the opposite bound, we write $\langle u_1,\cdot\rangle_E$ as the sum of two functionals $z_c$ and $z_s$ on $E$. We aim to choose $z_c$ to be continuous on $E$, and $z_s$ `small' in some sense. Then we have that
\begin{eqnarray*}
\fl(\mu_0\times\nu_0)(B^\delta(u_1,a_1)) &= e^{-\frac{1}{2}\|u_1\|_E^2}\int_{B^\delta(0,a_1)}e^{z_c(u) + z_s(u)}\,\mu_0(\dee u)\nu_0(\dee a)\\
&\leq \exp\left(-\frac{1}{2}\|u_1\|_E^2 +\delta\cdot\sup_{(u,a) \in B^1(0,a_1)} z_c(u)\right)\cdot\\
&\hspace{2cm}\left[(\mu_0\times\nu_0)(B^\delta(0,a_1)) + \int_{B^\delta(0,a_1)}(e^{z_s(u)} - 1)\,\mu_0(\dee u)\nu_0(\dee a)\right]
\end{eqnarray*}
where we have used the linearity of $z_c$ to extract $\delta$ from the supremum. As in \cite{lifshits}, using a result from \cite{sidak}, a special case of the Gaussian correlation conjecture, it follows that for any $C \in \mathbb{R}$ and any convex set $B\subseteq X$ symmetric about $0$,
\begin{eqnarray*}
\mu_0(B\cap\{u \in X\;|\;|z_s(u)|>C\}) \leq \mu_0(B)\mu_0(|z_s(\cdot)| > C).
\end{eqnarray*} 
Then for any increasing function $\varphi:\mathbb{R}_+\rightarrow\mathbb{R}_+$, one has
\begin{eqnarray*}
\fl\int_{B^\delta(0,a_1)}\varphi&(|z_s(u)|)\,\mu_0(\dee u)\nu_0(\dee a) = \int_{X\times\Lambda} \varphi(|z_s(u)|)\mathds{1}_{B^\delta(0,a_1)}(u,a)\,\mu_0(\dee u)\nu_0(\dee a)\\
\fl&=\int_0^\infty (\mu_0\times\nu_0)(\{(u,a) \in B^\delta(0,a_1)\;|\;\varphi(|z_s(u)|) > t\})\,\dee t\\
\fl&=\int_0^\infty (\mu_0\times\nu_0)(\{(u,a) \in B^\delta(0,a_1)\;|\;|z_s(u)| > \varphi^{-1}(t)\})\,\dee t\\
\fl&=\int_0^\infty \int_\Lambda \mu_0(\{u \in X\;|\;(u,a) \in B^\delta(0,a_1), |z_s(u)| > \varphi^{-1}(t)\})\,\nu_0(\dee a)\dee t\\
\fl&\leq \int_0^\infty \int_\Lambda \mu_0(\{u \in X\;|\;(u,a) \in B^\delta(0,a_1)\})\mu_0(|z_s(\cdot)| > \varphi^{-1}(t))\,\nu_0(\dee a)\dee t\\
\fl&= \int_0^\infty \mu_0(|z_s(\cdot)| > \varphi^{-1}(t))\underbrace{\left(\int_\Lambda \mu_0(\{u \in X\;|\;(u,a) \in B^\delta(0,a_1)\})\,\nu_0(\dee a)\right)}_{(\mu_0\times\nu_0)(B^\delta(0,a_1))}\dee t\\
\fl&= (\mu_0\times\nu_0)(B^\delta(0,a_1))\int_0^\infty \mu_0(|z_s(\cdot)| > \varphi^{-1}(t))\,\dee t\\
\fl&= (\mu_0\times\nu_0)(B^\delta(0,a_1))\int_0^\infty \mu_0(\varphi(|z_s(\cdot)|) > t)\,\dee t\\
\fl&= (\mu_0\times\nu_0)(B^\delta(0,a_1))\int_{X} \varphi(|z_s(u)|)\,\mu_0(\dee u).
\end{eqnarray*}
Choosing $\varphi(\cdot) = \exp(\cdot) - 1$ in this formula gives
\begin{eqnarray*}
\fl\int_{B^\delta(0,a_1)}(e^{|z_s(u)|} - 1)\,\mu_0(\dee u)\nu_0(\dee a) \leq (\mu_0\times\nu_0)(B^\delta(0,a_1))\int_{X} (e^{|z_s(u)|} - 1)\,\mu_0(\dee u).
\end{eqnarray*}
The space of linear measurable functionals on $E$, which contains $\langle u_1,\cdot\rangle_E$, is the $L^2$ closure of $E^*$. Thus for any $\varepsilon > 0$, the functionals $z_c, z_s$ can be chosen in order that the first of them is continuous and the second of them satisfies the inequality
\begin{eqnarray*}
\int_{X} (e^{|z_s(u)|} - 1)\,\mu_0(\dee u) \leq \varepsilon.
\end{eqnarray*}
It follows that for each $\varepsilon > 0$ we have
\begin{eqnarray}
\notag
\fl(\mu_0\times\nu_0)&(B^\delta(u_1,a_1))\\
\label{thm1upper}
\fl&\leq \exp\left(-\frac{1}{2}\|u_1\|_E^2 +\delta\cdot\sup_{(u,a) \in B^1(0,a_1)} z_c(u)\right)(\mu_0\times\nu_0)(B^\delta(0,a_1))(1+\varepsilon).
\end{eqnarray}
Since balls are bounded, $\varepsilon > 0$ is arbitrary and $z_c$ is continuous, we can combine (\ref{thm1lower}) and (\ref{thm1upper}) to deduce that there exists $M > 0$ such that
\begin{eqnarray*}
\fl e^{-\frac{1}{2}\|u_1\|_E^2}(\mu_0\times\nu_0)(B^\delta(0,a_1)) \leq(\mu_0\times\nu_0)(B^\delta(u_1,a_1)) &\leq e^{-\frac{1}{2}\|u_1\|_E^2 +M\delta}(\mu_0\times\nu_0)(B^\delta(0,a_1)).
\end{eqnarray*}
Now looking at the ratio of measures we see
\begin{eqnarray*}
\lim_{\delta\downarrow 0}\frac{(\mu_0\times\nu_0)(B^\delta(u_1,a_1))}{(\mu_0\times\nu_0)(B^\delta(u_2,a_2))} = e^{\frac{1}{2}\|u_2\|_E^2 - \frac{1}{2}\|u_1\|_E^2}\cdot\lim_{\delta\downarrow 0}\frac{(\mu_0\times\nu_0)(B^\delta(0,a_1))}{(\mu_0\times\nu_0)(B^\delta(0,a_2))}.
\end{eqnarray*}
We now deal with the geometric parameters. Let $a^*\in \intr(S)$ so that $\rho$ is positive in a neighbourhood of $a^*$ (we may take $a^* = a_1$ or $a_2$ since we assume they lie in $\intr(S)$). Then
\begin{eqnarray*}
\frac{(\mu_0\times\nu_0)(B^\delta(0,a_1))}{(\mu_0\times\nu_0)(B^\delta(0,a_2))} &= \frac{\int_{B^\delta(0,a_1)}\rho(a)\mu_0(\dee u)\dee a}{\int_{B^\delta(0,a_2)}\rho(a)\mu_0(\dee u)\dee a}\\
&= \frac{\int_{B^\delta(0,a^*)}\rho(a+a_1-a^*)\,\mu_0(\dee u)\dee a}{\int_{B^\delta(0,a^*)}\rho(a+a_2-a^*)\,\mu_0(\dee u)\dee a}\\
&= \frac{\int_{B^\delta(0,a^*)}\frac{\rho(a+a_1-a^*)}{\rho(a)}\,\mu_0(\dee u)\nu_0(\dee a)}{\int_{B^\delta(0,a^*)}\frac{\rho(a+a_2-a^*)}{\rho(a)}\,\mu_0(\dee u)\nu_0(\dee a)}.
\end{eqnarray*}
For sufficiently small $\delta$ both of the integrands are continuous.  A mean-value property hence holds for the integrals, and so we may divide both the numerator and denominator by $(\mu_0\times\nu_0)(B^\delta(0,a^*))$ and take limits to obtain
\begin{eqnarray*}
\lim_{\delta\downarrow 0}\frac{\displaystyle(\mu_0\times\nu_0)(B^\delta(0,a_1))}{\displaystyle(\mu_0\times\nu_0)(B^\delta(0,a_2))} &= \frac{\displaystyle \frac{\rho(a+a_1-a^*)}{\rho(a)}\bigg|_{a = a^*}}{\displaystyle \frac{\rho(a+a_2-a^*)}{\rho(a)}\bigg|_{a = a^*}}\\
&= \frac{\rho(a_1)}{\rho(a_2)}.
\end{eqnarray*}
We conclude that
\begin{eqnarray*}
\lim_{\delta\downarrow 0} \frac{(\mu_0\times\nu_0)(B^\delta(u_1,a_1))}{(\mu_0\times\nu_0)(B^\delta(u_2,a_2))} &= e^{\frac{1}{2}\|u_2\|_E^2 - \frac{1}{2}\|u_1\|_E^2}\cdot\frac{\rho(a_1)}{\rho(a_2)}\\
&= \exp\left(J(u_2) + K(a_2) - J(u_1) - K(a_1)\right).
\end{eqnarray*}
\end{proof}

\begin{lemma}
\label{lem:omcts}
Let $f,g:\Lambda\rightarrow\mathbb{R}$ be continuous, and $(u_1,a_1), (u_2,a_2) \in E\times\intr(S)$. Then
\begin{eqnarray*}
\lim_{\delta\downarrow 0}\frac{\int_{B^\delta(u_1,a_1)}f(a)\,\mu_0(\dee u)\nu_0(\dee a)}{\int_{B^\delta(u_2,a_2)}g(a)\,\mu_0(\dee a)\nu_0(\dee a)} = e^{\frac{1}{2}\|u_2\|_E^2 - \frac{1}{2}\|u_1\|_E^2}\cdot\frac{\rho(a_1)}{\rho(a_2)}\cdot\frac{f(a_1)}{g(a_2)}.
\end{eqnarray*}
\end{lemma}

\begin{proof}
Let $\eps > 0$. Then by the continuity of $f$ and $g$, and the assumption on the norm on $X\times\Lambda$, there exists $\delta > 0$ such that
\begin{eqnarray*}
\frac{(f(a_1)-\eps)(\mu_0\times\nu_0)(B^\delta(u_1,a_1))}{(g(a_2)+\eps)(\mu_0\times\nu_0)(B^\delta(u_2,a_2))} &\leq \frac{\int_{B^\delta(u_1,a_1)}f(a)\,\mu_0(\dee u)\nu_0(\dee a)}{\int_{B^\delta(u_2,a_2)}g(a)\,\mu_0(\dee u)\nu_0(\dee a)}\\
&\leq \frac{(f(a_1)+\eps)(\mu_0\times\nu_0)(B^\delta(u_1,a_1))}{(g(a_2)-\eps)(\mu_0\times\nu_0)(B^\delta(u_2,a_2))}.
\end{eqnarray*}
The result now follows by the previous lemma.
\end{proof}

\begin{proof}[Proof of Theorem \ref{omthm}]
Let $(u_1,a_1), (u_2,a_2) \in E\times\intr(S)$. The case $\Phi\equiv 0$ is the result of Lemma \ref{ombasic}. Now proceeding analagously to \cite{map},
\begin{eqnarray*}
\fl \frac{\mathcal{J}^\delta(u_1,a_1)}{\mathcal{J}^\delta(u_2,a_2)} &= \frac{\int_{B^\delta(u_1,a_1)}\exp(-\Phi(u,a))\,\mu_0(\dee u)\nu_0(\dee a)}{\int_{B^\delta(u_2,a_2)}\exp(-\Phi(u,a))\,\mu_0(\dee u)\nu_0(\dee a)}\\
\fl &=\frac{\int_{B^\delta(u_1,a_1)}\exp(-\Phi(u,a)+\Phi(u_1,a_1))\exp(-\Phi(u_1,a_1))\,\mu_0(\dee u)\nu_0(\dee a)}{\int_{B^\delta(u_2,a_2)}\exp(-\Phi(u,a)+\Phi(u_2,a_2))\exp(-\Phi(u_2,a_2))\,\mu_0(\dee u)\nu_0(\dee a)}.
\end{eqnarray*}
Using Assumptions \ref{assump:map}(iv), we have that for any $(u,a),(v,b) \in X\times\Lambda$,
\begin{eqnarray*}
|\Phi(u,a) - \Phi(v,b)| \leq M_3(r,a)\|u-v\|_X + |\Phi(v,a) - \Phi(v,b)|
\end{eqnarray*}
where $r > \max\{\|u\|_X,\|v\|_X\}$. Now set
\begin{eqnarray*}
L_1 = \max_{|a| < |a_1| + \delta} M_3(\|u_1\|_X + \delta,a),\\
L_2 = \max_{|a| < |a_2| + \delta} M_3(\|u_2\|_X + \delta,a),
\end{eqnarray*}
which are finite due to the continuity assumption on $M_3$. Then 
\begin{eqnarray*}
\frac{\mathcal{J}^\delta(u_1,a_1)}{\mathcal{J}^\delta(u_2,a_2)} &\leq e^{\delta(L_1+L_2)}e^{-\Phi(u_1,a_1)+\Phi(u_2,a_2))}\\
&\hspace{1cm}\times\frac{\int_{B^\delta(u_1,a_1)}\exp(|\Phi(u_1,a) - \Phi(u_1,a_1)|)\,\mu_0(\dee u)\nu_0(\dee a)}{\int_{B^\delta(u_2,a_2)}\exp(-|\Phi(u_2,a) - \Phi(u_2,a_2)|)\,\mu_0(\dee u)\nu_0(\dee a)}.
\end{eqnarray*}
Note that both integrands are continuous in $a$, and so we may use the previous lemma. Taking $\limsup_{\delta\downarrow 0}$ of both sides gives
\begin{eqnarray*}
\limsup_{\delta\downarrow 0} \frac{\mathcal{J}^\delta(u_1,a_1)}{\mathcal{J}^\delta(u_2,a_2)} \leq e^{-I(u_1,a_1)+I(u_2,a_2)}.
\end{eqnarray*}
A similar method gives that the $\liminf_{\delta\downarrow 0}$ is bounded below by the RHS and so we have that for any $(u_1,a_2), (u_2,a_2) \in E\times\intr(S)$,
\begin{eqnarray*}
\lim_{\delta\downarrow 0}\frac{\mathcal{J}^\delta(u_1,a_1)}{\mathcal{J}^\delta(u_2,a_2)} = e^{I(u_2,a_2) - I(u_1,a_1)}.
\end{eqnarray*}
Noting that $I$ is continuous on $E\times S$, we see that $I$ agrees with the Onsager-Machlup functional on $E\times S$. Finally note that $I(u,a) = \infty$ on $(X\setminus E)\times\Lambda$ and $E\times(\Lambda\setminus S)$.
\end{proof}

\begin{remark}
\label{separate}
Note that the limit above is independent of the choice of norm used on the product space $X\times\Lambda$ when referring to the balls. If we use the norm given by
\begin{eqnarray*}
\|(x,a)\| = \max\{\|x\|_X,|a|\}
\end{eqnarray*}
then we have that
\begin{eqnarray*}
B^\delta(u,a) = B^\delta(u)\times B^\delta(a)
\end{eqnarray*}
and so may deduce that, for any choice of norm on $X\times\Lambda$,
\begin{eqnarray*}
\lim_{\delta\downarrow 0}\frac{\displaystyle(\mu_0\times\nu_0)(B^\delta(u_1,a_1))}{\displaystyle(\mu_0\times\nu_0)(B^\delta(u_2,a_2))} &= \lim_{\delta\downarrow 0}\frac{\displaystyle(\mu_0\times\nu_0)(B^\delta(u_1)\times B^\delta(a_1))}{\displaystyle(\mu_0\times\nu_0)(B^\delta(u_2)\times B^\delta(a_2))}\\
&=\lim_{\delta\downarrow 0}\frac{\mu_0(B^\delta(u_1))}{\mu_0(B^\delta(u_2))}\cdot \frac{\nu_0(B^\delta(a_1))}{\nu_0(B^\delta(a_2))}.
\end{eqnarray*}
This will be useful later for separating integrals.
\end{remark}

\begin{proof}[Proof of Theorem \ref{exist}]
We follow the idea of the proof of Theorem 5.4 in \cite{inverse}, which is based on \cite{dacorogna} and \cite{kinder}, and first show $I = \Phi + J + K$ is weakly lower semicontinuous on $E\times S$. Let $(u_n,a_n) \rightharpoonup (\bar{u},\bar{a})$ in $E\times S$. Since $S\subseteq \mathbb{R}^k$, weak convergence of the second component is equivalent to strong convergence. Since $\mu_0(X) = 1$, $E$ is compactly embedded in $X$ and so $u_n\rightarrow\bar{u}$ strongly in $X$. In the proof of existence of the posterior distribution we showed that $\Phi$ is continuous on $X\times\Lambda$, and so we deduce that $\Phi(u_n,a_n)\rightarrow\Phi(u,a)$. Hence $\Phi$ is weakly continuous on $E\times S$. The functional $J$ is weakly lower semicontinuous on $E$ and $K$ is continuous on $S$, and so $I$ is weakly lower semicontinuous on $E\times S$.

Now we show $I$ is coercive on $E\times S$. Since $E$ is compactly embedded in $X$ there exists a $C > 0$ such that
\begin{eqnarray*}
\|u\|^2_{X} \leq C\|u\|_E^2.
\end{eqnarray*}
Therefore by Assumption \ref{assump:map}(i) it follows that, for any $\varepsilon > 0$, there is an $M(\varepsilon) \in \mathbb{R}$ such that
\begin{eqnarray*}
I(u,a) \geq M(\varepsilon) + \left(\frac{1}{2} - C\varepsilon\right)\|u\|_E^2 + K(a).
\end{eqnarray*}
Since $K$ is bounded below\footnote{Recall in subsection \ref{ssec:priorgeo} we assumed $\rho$ to be continuous on the compact set $S$, and hence bounded.} by $-\log\|\rho\|_{\infty}$, we may incorporate this into the constant term $M(\varepsilon)$: 
\begin{eqnarray*}
I(u,a) \geq \widetilde{M}(\varepsilon) + \left(\frac{1}{2} - C\varepsilon\right)\|u\|_E^2.
\end{eqnarray*}
By choosing $\varepsilon = 1/4C$, we see that there is an $M \in \mathbb{R}$ such that, for all $(u,a) \in E\times S$,
\begin{eqnarray*}
I(u,a) \geq \frac{1}{4}\|u\|^2_E + M
\end{eqnarray*}
which establishes coercivity.

Now take a minimising sequence $(u_n,a_n)$ such that for any $\delta > 0$ there exists an $N_1 = N_1(\delta)$ such that
\begin{eqnarray*}
M \leq \bar{I} \leq I(u_n,a_n) \leq \bar{I} + \delta,\;\;\;\forall n\geq N_1.
\end{eqnarray*}
From the coercivity it can be seen that the sequence $(u_n,a_n)$ is bounded in $E\times S$. Since $E\times S$ is a closed subset of a Hilbert space, there exists $(\bar{u},\bar{a}) \in E\times S$ such that (possibly along a subsequence) $(u_n,a_n)\rightharpoonup (\bar{u},\bar{a})$ in $E\times S$. From the weak lower semicontinuity of $I$ it follows that, for any $\delta > 0$,
\begin{eqnarray*}
\bar{I} \leq I(\bar{u},\bar{a}) \leq \bar{I} + \delta.
\end{eqnarray*}
Since $\delta$ is arbitrary the first result follows.

Now consider the subsequence $(u_n,a_n)\rightharpoonup (\bar{u},\bar{a})$. The convergence of $a_n\rightarrow \bar{a}$ is strong, so all that needs to be checked is that $u_n\rightarrow\bar{u}$ strongly in $X$. This follows from exactly the same argument as in the proof of Theorem 5.4 in \cite{inverse} (taking $\bar{a}$ as the second parameter in $I$ and $\Phi$) and so the second result follows.
\end{proof}

Before proving Theorem \ref{thm:lim_is_map} we first collect some results on centred Gaussian measures from \cite{map}, specifically Lemmas 3.6, 3.7, and 3.9. For $u \in X$, let
\begin{eqnarray*}
\mathcal{J}^\delta_0(u) = \mu_0(B^\delta(u)).
\end{eqnarray*}

\begin{proposition}
\label{gausslemma}
\begin{enumerate}[(i)]
\item Let $\delta > 0$ and $u \in X$. Then we have
\begin{eqnarray*}
\frac{\mathcal{J}^\delta_0(u)}{\mathcal{J}^\delta_0(0)} \leq ce^{-\frac{a_1}{2}(\|u\|_{X} - \delta)^2}
\end{eqnarray*}
where $c = \exp\left(\frac{a_1}{2}\delta^2\right)$ and $a_1$ is a constant independent of $z$ and $\delta$.
\item Suppose that $\bar{u} \notin E$, $(u^\delta)_{\delta > 0}\subseteq X$ and $u^\delta$ converges weakly to $\bar{u} \in X$ as $\delta\downarrow 0$. Then for any $\varepsilon > 0$ there exists $\delta$ small enough such that
\begin{eqnarray*}
\frac{\mathcal{J}^\delta_0(u^\delta)}{\mathcal{J}^\delta_0(0)} < \varepsilon.
\end{eqnarray*}
%\item Suppose $u \notin E$. Then
%\begin{eqnarray*}
%\lim_{\delta\downarrow 0}\frac{\mathcal{J}^\delta_0(u)}{\mathcal{J}^\delta_0(0)} = 0.
%\end{eqnarray*}
\item Consider $(u^\delta)_{\delta > 0} \subseteq X$ and suppose that $u^\delta$ converges weakly and not strongly to $0$ in $X$ as $\delta\downarrow 0$. Then for any $\varepsilon > 0$ there exists $\delta$ small enough such that
\begin{eqnarray*}
\frac{\mathcal{J}^\delta_0(u^\delta)}{\mathcal{J}^\delta_0(0)} < \varepsilon.
\end{eqnarray*}
\end{enumerate}
\end{proposition}

\begin{proof}[Proof of Theorem \ref{thm:lim_is_map}]
\begin{enumerate}[(i)]
\item We first show $(u^\delta,a^\delta)$ is bounded in $X\times\Lambda$. The boundedness of the second component is clear since $S$ is bounded, so it suffices to show that $(u^\delta)$ is bounded in $X$. This is proved in the same way as in Theorem 3.5 in \cite{map}.

In the proof of existence of the posterior measure, Theorem \ref{thm:post_exist}, we show that if $r> 0$ and $\|u\|_X,|a| < r$, then there exists $K(r) > 0$ such that $\Phi(u,a) \leq K(r)$. Letting $c = e^M e^{-K(1)} > 0$, it follows in the same was as \cite{map} that, given any $a \in S$, for $\delta < 1$ we have
\begin{eqnarray*}
\mathcal{J}^\delta_0(u^\delta,a) \geq c \mathcal{J}^\delta_0(0,a).
\end{eqnarray*}
Suppose that $(u^\delta)$ is not bounded in $X$ so that for any $R>0$ there exists $\delta_R$ such that $\|u^{\delta_R}\|_{X} > R$, with $\delta_R\rightarrow 0$ as $R\rightarrow\infty$. Then the above bound says that
\begin{eqnarray*}
\frac{\mathcal{J}^\delta_0(u^\delta,a)}{\mathcal{J}^\delta_0(0,a)} = \frac{\mu_0(B^\delta(u^\delta))}{\mu_0(B^\delta(0))}\cdot\frac{\nu_0(B^\delta(a))}{\nu_0(B^\delta(a))} \geq c.
\end{eqnarray*}
This contradicts Proposition \ref{gausslemma}(i) above. Therefore there exists $R,\delta_R > 0$ such that
\begin{eqnarray*}
\|(u^\delta,a^\delta)\|_{X\times\Lambda} \leq R\;\;\;\text{for any }\delta < \delta_R.
\end{eqnarray*}
Hence there exist $(\bar{u},\bar{a}) \in X\times\Lambda$ and a subsequence of $(u^\delta,a^\delta)_{0<\delta<\delta_R}$ which converges weakly in $X\times\Lambda$ to $(\bar{u},\bar{a})$ as $\delta\downarrow 0$. For simplicity of notation we still call this subsequence $(u^\delta,a^\delta)$.

We now show that $(u^\delta,a^\delta)$ converges strongly to an element of $E\times S$. We first show that $(\bar{u},\bar{a}) \in X\times S$.

Note that any limit point of $a^\delta$ must lie in $S$. Suppose it did not, and a limit point was $a^* \notin S$. Then there exists $\delta^\dagger > 0$ such that along a subsequence converging to $a^*$, $\delta < \delta^\dagger$ implies $a^\delta \notin S$ since $S$ is closed. For $\delta < \frac{1}{2}\mathrm{dist}(a^*,S)\wedge\delta^\dagger$ we then have $B^\delta(a^\delta)\cap S = \varnothing$. In particular $\nu_0(B^\delta(a^\delta)) = 0$ for all such $\delta$, which in turn implies $\mathcal{J}^{\delta}(u,a^\delta) = 0$ for any $u \in X$ contradicting the definition of $a^\delta$. It follows that we must have $\bar{a} \in S$.

We need to show $\bar{u} \in E$. From the definition of $(u^\delta,a^\delta)$ and the bounds on $\Phi$ we have for $\delta$ small enough and some\footnote{Remark \ref{separate} tells us that we can separate the integrals in the limit $\delta\downarrow 0$.} $\alpha$ close to 1,
\begin{eqnarray*}
1 \leq \frac{\mathcal{J}^\delta(u^\delta,0)}{\mathcal{J}^\delta(0,0)} &\leq \alpha\frac{e^{-M}\int_{B^\delta(u^\delta)}\mu_0(\dee u)\int_{B^\delta(0)}\nu_0(\dee a)}{e^{-K(1)}\int_{B^\delta(0)}\mu_0(\dee u)\int_{B^\delta(0)}\nu_0(\dee a)}\\
&= \alpha e^{K(1)-M}\frac{\int_{B^\delta(u^\delta)}\mu_0(\dee u)}{\int_{B^\delta(0)}\mu_0(\dee u)}.
\end{eqnarray*}
We use Proposition \ref{gausslemma}(ii). Supposing $\bar{u}\notin E$, for any $\varepsilon > 0$ there exists $\delta$ small enough such that
\begin{eqnarray*}
\frac{\int_{B^\delta(u^\delta)}\mu_0(\dee u)}{\int_{B^\delta(0)}\mu_0(\dee u)} < \varepsilon.
\end{eqnarray*}
We may choose $\varepsilon = \frac{1}{2\alpha}e^{M-K(1)}$ to deduce that there exists $\delta$ small enough such that
\begin{eqnarray*}
1 \leq \frac{\mathcal{J}^\delta(u^\delta,0)}{\mathcal{J}^\delta(0,0)} < \frac{1}{2}
\end{eqnarray*}
which is a contradiction, and so $\bar{u} \in E$.

Knowing that $(\bar{u},\bar{a}) \in E\times S$ we now show that the convergence is strong. Any convergence of the second component will be strong and so we just need to show that $u^{\delta}\rightarrow \bar{u}$ strongly in $X$. Suppose the convergence is not strong, then we may use Proposition \ref{gausslemma}(iii) on the sequence $u^\delta - \bar{u}$. The same choice of $\varepsilon$ as above leads to the same contradiction, and so we deduce that $\bar{u}\rightarrow\bar{u}$ strongly in $X$ and the first result is proved.

\item We now show that $(\bar{u},\bar{a})$ is a MAP estimator and minimises $I$. As in \cite{map}, and the proof of Theorem \ref{omthm}, we can use Assumptions \ref{assump:map}(iii) to see that
\begin{eqnarray*}
\frac{\mathcal{J}^\delta(u^\delta,a^\delta)}{\mathcal{J}^\delta(\bar{u},\bar{a})} &\leq e^{\delta(L_1+L_2)}e^{-\Phi(u^\delta,a^\delta)+\Phi(\bar{u},\bar{a}))}\\
&\hspace{1cm}\times\frac{\int_{B^\delta(u^\delta,a^\delta)}\exp(|\Phi(u^\delta,a) - \Phi(u^\delta,a^\delta)|)\,\mu_0(\dee u)\nu_0(\dee a)}{\int_{B^\delta(\bar{u},\bar{a})}\exp(-|\Phi(\bar{u},a) - \Phi(\bar{u},\bar{a})|)\,\mu_0(\dee u)\nu_0(\dee a)}
\end{eqnarray*}
where
\begin{eqnarray*}
L_1 &= \max_{|a|\leq |a_1|+\delta} M_3(\|u^\delta\|_{X} + \delta,a),\\
L_2 &= \max_{|a|\leq |a_2|+\delta} M_3(\|\bar{u}\|_{X}+\delta,a).
\end{eqnarray*}
Therefore using the continuity of $\Phi$, as shown in the proof of existence of the posterior distribution, and that $(u^\delta,a^\delta)\rightarrow(\bar{u},\bar{a})$ strongly in $X\times\Lambda$,
\begin{eqnarray*}
\limsup_{\delta\downarrow 0}\frac{\mathcal{J}^\delta(u^\delta,a^\delta)}{\mathcal{J}^\delta(\bar{u},\bar{a})} \leq \limsup_{\delta\downarrow 0} \frac{\int_{B^\delta(u^\delta,a^\delta)}\mu_0(\dee u)\nu_0(\dee a)}{\int_{B^\delta(\bar{u},\bar{a})}\mu_0(\dee u)\nu_0(\dee a)}.
\end{eqnarray*}
Suppose $u^\delta$ is not bounded in $E$, or if it is, it only converges weakly (and not strongly) in $E$. Then $\|\bar{u}\|_E < \liminf_{\delta\downarrow 0}\|u^\delta\|_E$, and hence for small enough $\delta$, $\|\bar{u}\|_E < \|u^\delta\|_E$. Therefore, since $\mu_0$ is centered and $\|u^\delta - \bar{u}\|_{X}\rightarrow 0$, $|a^\delta - \bar{a}|\rightarrow 0$,
\begin{eqnarray*}
\fl\limsup_{\delta\downarrow 0} \frac{\int_{B^\delta(u^\delta,a^\delta)}\mu_0(\dee u)\nu_0(\dee a)}{\int_{B^\delta(\bar{u},\bar{a})}\mu_0(\dee u)\nu_0(\dee a)} &= \limsup_{\delta\downarrow 0} \frac{\int_{B^\delta(u^\delta)}\mu_0(\dee u)\int_{B^\delta(a^\delta)}\nu_0(\dee a)}{\int_{B^\delta(\bar{u})}\mu_0(\dee u)\int_{B^\delta(\bar{a})}\nu_0(\dee a)}\\
&\leq \limsup_{\delta\downarrow 0} \frac{\int_{B^\delta(u^\delta)}\mu_0(\dee u)}{\int_{B^\delta(\bar{u})}\mu_0(\dee u)}\cdot\limsup_{\delta\downarrow 0}\frac{\int_{B^\delta(a^\delta)}\nu_0(\dee a)}{\int_{B^\delta(\bar{a})}\nu_0(\dee a)}\\
&\leq \limsup_{\delta\downarrow 0}\frac{\int_{B^\delta(a^\delta)}\nu_0(\dee a)}{\int_{B^\delta(\bar{a})}\nu_0(\dee a)}\\
&= \limsup_{\delta\downarrow 0}\frac{\frac{1}{|B^\delta(a^\delta)|}\int_{B^\delta(a^\delta)}\rho(a)\,\dee a}{\frac{1}{|B^\delta(\bar{a})|}\int_{B^\delta(\bar{a})}\rho(a)\,\dee a}\\
&= 1.
\end{eqnarray*}
The final equality above follows from the continuity of the integrand and the fact that $|a^\delta-\bar{a}|\rightarrow 0$: both the numerator and the denominator tend to $\rho(\bar{a})$.

Since by definition of $(u^\delta,a^\delta)$, $\mathcal{J}^\delta(u^\delta,a^\delta) \geq \mathcal{J}^\delta(\bar{u},\bar{a})$ and hence
\begin{eqnarray*}
\liminf_{\delta\downarrow 0}\frac{\mathcal{J}^\delta(u^\delta,a^\delta)}{\mathcal{J}^\delta(\bar{u},\bar{a})} \geq 1,
\end{eqnarray*}
this implies that
\begin{eqnarray}
\label{maplimit}
\lim_{\delta\downarrow 0} \frac{\mathcal{J}^\delta(u^\delta,a^\delta)}{\mathcal{J}^\delta(\bar{u},\bar{a})} = 1.
\end{eqnarray}

In the case where $(u^\delta)$ converges strongly to $\bar{u}$ in $E$, we see from the proof of Lemma \ref{ombasic} that we have
\begin{eqnarray*}
\fl e^{\frac{1}{2}\|\bar{u}\|_E^2 - \frac{1}{2}\|u^\delta\|_E^2 - M\delta}\frac{(\mu_0\times\nu_0)(B^\delta(0,a^\delta))}{(\mu_0\times\nu_0)(B^\delta(0,\bar{a}))} &\leq \frac{(\mu_0\times\nu_0)(B^\delta(u^\delta,a^\delta))}{(\mu_0\times\nu_0)(B^\delta(\bar{u},\bar{a}))}\\
&\leq e^{\frac{1}{2}\|\bar{u}\|_E^2 - \frac{1}{2}\|u^\delta\|_E^2 + M\delta}\frac{(\mu_0\times\nu_0)(B^\delta(0,a^\delta))}{(\mu_0\times\nu_0)(B^\delta(0,\bar{a}))}.
\end{eqnarray*}
Since we have $u^\delta\rightarrow\bar{u}$ strongly in $E$ we have in particular that $\|u^\delta\|_E\rightarrow\|\bar{u}\|_E$. It follows that $e^{\frac{1}{2}\|\bar{u}\|_E^2 - \frac{1}{2}\|u^\delta\|_E^2 \pm M\delta} \rightarrow 1$ as $\delta\downarrow 0$. Now using the continuity of $\rho$ and the fact that $|a^\delta - \bar{a}|\rightarrow 0$, an argument similar to that in the proof of Lemma \ref{ombasic} shows that
\begin{eqnarray*}
\lim_{\delta\downarrow 0} \frac{(\mu_0\times\nu_0)(B^\delta(0,a^\delta)}{(\mu_0\times\nu_0)(B^\delta(0,\bar{a}))} = 1.
\end{eqnarray*}
We therefore deduce that
\begin{eqnarray*}
\lim_{\delta\downarrow 0} \frac{\int_{B^\delta(u^\delta,a^\delta)}\mu_0(\dee u)\nu_0(\dee a)}{\int_{B^\delta(\bar{u},\bar{a})}\mu_0(\dee u)\nu_0(\dee a)} = 1
\end{eqnarray*}
and (\ref{maplimit}) follows again. Therefore $(\bar{u},\bar{a})$ is a MAP estimator of the measure $\mu$.

The proof that $(\bar{u},\bar{a})$ minimises $I$ is identical to that in the proof of Theorem 3.5 in \cite{map}.
\end{enumerate}
\end{proof}

\ack{The authors would like to thank Shiwei Lan and Claudia Schillings for helpful discussions on the adjoint method used in the minimisation procedure. The authors would also like to thank Marco Iglesias for more general discussions. MMD is supported by EPSRC grant EP/H023364/1 as part of the MASDOC DTC at the University of Warwick. AMS is supported by EPSRC and ONR. This research utilised Queen Mary's MidPlus computational facilities, supported by QMUL Research-IT and funded by EPSRC grant EP/K000128/1.\\}

\section*{References}
\bibliographystyle{myplain}
\bibliography{references}

\end{document}